\documentclass[11pt]{article}

\usepackage[top=1in, bottom=1in, left=1in, right=1in]{geometry}

\usepackage{amsmath}
\usepackage{amsthm}
\usepackage{amssymb}
\usepackage{mathtools}
\usepackage{mathrsfs}
\usepackage{tocloft}
\usepackage{setspace}
\usepackage[toc,page]{appendix}

\allowdisplaybreaks

\newtheorem{lemma}{Lemma}[section]
\newtheorem{proposition}{Proposition}[section]
\newtheorem{theorem}{Theorem}[section]
\newtheorem{corollary}{Corollary}[section]

\theoremstyle{definition}
\newtheorem{definition}{Definition}[section]

\theoremstyle{remark}
\newtheorem{remark}{Remark}[section]

\newcommand{\dist}{\operatorname{dist}}
\newcommand{\SL}{\operatorname{PSL}}
\newcommand{\arccot}{\operatorname{arccot}}

\begin{document}
\title{Sparse Equidistribution of Unipotent Orbits in Finite-Volume Quotients of $\SL(2,\mathbb R)$\\\vspace{0.3in} \large With Appendices of Hausdorff Dimensions of Subsets of Non-Diophantine\\ Points and the Effective Equidistribution of Horospherical Orbits}
\author{Cheng Zheng}
\date{}
\maketitle

\begin{abstract}
In this note, we consider the orbits $\{pu(n^{1+\gamma})|n\in\mathbb N\}$ in $\Gamma\backslash\SL(2,\mathbb R)$, where $\Gamma$ is a non-uniform lattice in $\SL(2,\mathbb R)$ and $\{u(t)\}$ is the standard unipotent one-parameter subgroup in $\SL(2,\mathbb R)$. Under a Diophantine condition on the intial point $p$, we can prove that the trajectory $\{pu(n^{1+\gamma})|n\in\mathbb N\}$ is equidistributed in $\Gamma\backslash\SL(2,\mathbb R)$ for small $\gamma>0$, which generalizes a result of Venkatesh \cite{V}. We will compute Hausdorff dimensions of subsets of non-Diophantine points in Appendix A, using results of lattice counting problem. In Appendix B we will use a technique of \cite{V} and an exponential mixing property to prove a weak version of a result of Str\"ombergsson \cite{S}, which is about the effective equidistribution of horospherical orbits.
\end{abstract}

\section{Introduction}
The theory of equidistribution of unipotent flows on homogeneous spaces has been studied extensively over the past few decades. Furstenberg \cite{F} first proved that the unipotent flow on $\Gamma\backslash\SL(2,\mathbb R)$, where $\Gamma$ is a uniform lattice, is uniquely ergodic. In \cite{Dani-Invent78} Dani classified ergodic invariant measures for unipotent flows on finite volume homogeneous spaces of $\SL(2,\mathbb R)$, and using this result Dani and Smillie \cite{DS} proved that any non-periodic unipotent orbit is equidistributed on $\Gamma\backslash\SL(2,\mathbb R)$ for any lattice $\Gamma$. The proof of the Oppenheim Conjecture due to Margulis~\cite{Margulis-Oppenheim1} by proving a special case of Raghunathan's conjecture drew a lot of attention to this subject. Soon afterwords, Ratner published her seminal work \cite{R1,R2,R3} proving measure classification theorem for unipotent actions on homogeneous spaces as conjectured by Raghunathan and Dani~\cite{Dani-Invent81}. Using these results, Ratner~\cite{
R4} proved  that any unipotent orbit in a finite volume homogeneous space is equidistributed in its orbit closure; see also Shah~\cite{Shah-MathAnn91} for the case of Rank-$1$ semisimple groups. 

Ratner's work has led to many new extensions and number theoretic applications of ergodic theory of unipotent flows. One of these results, which is related to this paper, was the work by Shah \cite{Sh}. In that paper, Shah asked whether $\{pu(n^2)|n\in\mathbb N\}$ is equidistributed in a sub-homogeneous space of $\SL(2,\mathbb Z)\backslash\SL(2,\mathbb R)$, where $u:\mathbb R\to\SL(2,\mathbb R)$ is the standard unipotent $1$-parameter subgroup $$u(t)=\left(\begin{array}{cc}1 & t \\0 & 1\end{array}\right).$$

In this direction, Venkatesh published a result about sparse equidistribution (\cite{V}, Theorem 3.1). There he introduced a soft technique of calculations by using a discrepancy trick, and proved that if $\Gamma$ is a cocompact lattice in $\SL(2,\mathbb R)$ and $\gamma>0$ is a small number depending on the spectral gap of the Laplacian on $\Gamma\backslash\SL(2,\mathbb R)$, then for any point $p\in\Gamma\backslash\SL(2,\mathbb R)$ we have $$\frac1N\sum\limits_{n=0}^{N-1}f(pu(n^{1+\gamma}))\to\int_{\Gamma\backslash\SL(2,\mathbb R)} fd\mu.$$ In other words, in the case of $\Gamma\backslash\SL(2,\mathbb R)$ being compact, the equidistribtion holds for the sparse subset $\{n^{1+\gamma}|n\in\mathbb N\}$. It is worth noting that recently Tanis and Vishe \cite{TV} improve some results of Venkatesh \cite{V} and they obtain an absolute constant $\gamma>0$ which does not depend on the spectral gap.

In this paper, we will consider the sparse subset $\{n^{1+\gamma}|n\in\mathbb N\}$ and orbits of $\{u(n^{1+\gamma})|n\in\mathbb N\}$ in $\Gamma\backslash\SL(2,\mathbb R)$, where $\Gamma$ is a non-uniform lattice. We want to prove a sparse equidistribution theorem similar to Shah's conjecture \cite{Sh} and the work of Venkatesh \cite{V} and that of Tanis and Vishe \cite{TV}. To deal with the complexity caused by initial points of unipotent orbits, we introduce a Diophantine condition for points in $\Gamma\backslash\SL(2,\mathbb R)$ as follows.

Let $G=\SL(2,\mathbb R)$ and we consider the Siegel sets $N_\Omega A_\alpha K$ where $$N_{\Omega}=\left\{\left(\begin{array}{cc}1 & s \\0 & 1\end{array}\right)\Big|s\text{ is in a bounded subset }\Omega\subset\mathbb R\right\},\quad A_\alpha=\left\{\left(\begin{array}{cc}s & 0 \\0 & s^{-1}\end{array}\right)\Big|s\geq\alpha\right\}$$ and $K=SO(2).$  For the non-uniform lattice $\Gamma$, there exist $\sigma_j\in G$ and bounded intervals $\Omega_j\subset\mathbb R$  $(1\leq j\leq k)$ with the following property (\cite{GR}, \cite{DS})

\begin{enumerate}
\item For some $\alpha>0$, $G=\bigcup\limits_{j=1}^k\Gamma\sigma_jN_{\Omega_j} A_\alpha K$.
\item $\sigma_j^{-1}\Gamma\sigma_j\cap N$ is a cocompact lattice in $N$.
\item $N_{\Omega_j}$ is a fundamental domain of $\sigma_j^{-1}\Gamma\sigma_j\cap N\backslash N$.
\end{enumerate}

We will fix $\sigma_j$ $(1\leq j\leq k)$ in such a way that in the upper half plane $\mathfrak H$, each $\sigma_j$ corresponds to a cusp $\eta_j$, i.e. $\lim_{t\to\infty}\sigma_j\cdot it=\eta_j$, and $\eta_1,\eta_2,\dots,\eta_k$ are the inequivalent cusps of $\Gamma\backslash\mathfrak H$. Let $\Gamma_j=\Gamma\cap\sigma_j N\sigma_j^{-1}.$ Let $\pi_j$ be the covering map $$\pi_j:\Gamma_j\backslash G\to\Gamma\backslash G.$$ Now consider the usual action of $G$ on $\mathbb R^2$ and let $e_1=\left(\begin{array}{c}1 \\0\end{array}\right)$. For each $j$, we can define a map $$m_j:\Gamma_j\backslash G\to\mathbb R^2/\pm$$ by $$ m_j(q)=g^{-1}\sigma_j e_1$$ for $q=\Gamma_j g\in\Gamma_j\backslash G$, where $\mathbb R^2/\pm$ means that we identify every $v\in\mathbb R^2$ with its opposite $-v$. In this way, we obtain $k$ maps $m_j$ $(j=1,2,\dots,k)$ whose images are all in $\mathbb R^2/\pm$. Using these notations, we can give the following definition of Diophantine condition of a point $p\in\Gamma\backslash G$. 

\begin{definition}
Let $p\in\Gamma\backslash G$. We say that $p$ is Diophantine of type $(\kappa_1,\kappa_2,\dots,\kappa_k)$ for some $\kappa_j>0$ $(j=1,2,\dots,k)$ if for each $j$, there exist $\mu_j,\nu_j>0$ such that for every point $\left(\begin{array}{c}a \\b\end{array}\right)\in m_j(\pi_j^{-1}(p))$, we have either $|b|\geq\mu_j$ or $|a|^{\kappa_j}|b|\geq\nu_j$.
\end{definition}
\begin{remark}
This notion of Diophantine type on $p\in\Gamma\backslash G$ has been studied well in an equivalent form; it can be connected to the excursion rate of the geodesic orbit $\{g_t(p)\}_{t>0}$. We will prove this in section 3.
\end{remark}

It is straightforward to verify that if $g\in AN$ then the Diophantine types of $p$ and $pg$ are the same; although the choices of $\mu_j,\nu_j>0$ in the above definition may differ. The hausdorff dimension of the complement of the set of points of the Diophantine type $(\kappa_1,\kappa_2,\dots,\kappa_k)$ will be discussed in section 7. We will see that almost every point satisfies the Diophantine condition of type $(\kappa_1,\kappa_2,\dots,\kappa_k)$ when $\kappa_1,\kappa_2,\dots,\kappa_k>1$. When $\min\{\kappa_1,\kappa_2,\dots,\kappa_k\}=1$, the set of points of the Diophantine type $(\kappa_1,\kappa_2,\dots,\kappa_k)$ has zero Haar measure but has full Hausdorff dimension.

Now we state the main theorem in this paper.

\begin{theorem}[Main theorem]\label{main theorem}
Let $\Gamma$ be a non-uniform lattice in $\SL(2,\mathbb R)$ and $k$ the number of inequivalent cusps of $\Gamma\backslash\SL(2,\mathbb R)$. Suppose that $p\in\Gamma\backslash\SL(2,\mathbb R)$ is Diophantine of type $(\kappa_1,\kappa_2,\dots,\kappa_k).$ Then there exists a constant $\gamma_0>0$ such that for any $0<\gamma<\gamma_0$, we have $$\frac1N\sum\limits_{n=0}^{N-1}f(pu(n^{1+\gamma}))\to\int_{\Gamma\backslash\SL(2,\mathbb R)}fd\mu.$$ Here the constant $\gamma_0$ depends on $\kappa_1,\kappa_2,\dots,\kappa_k$ and $\Gamma$, and $f$ is any bounded continuous function on $\Gamma\backslash\SL(2,\mathbb R)$.
\end{theorem}
\begin{remark}\label{remark}
From the proof of the main theorem, we will see that the constant $$\gamma_0=\min\left\{\frac{s^2}{(s+4)(\kappa_j+4)}\Big|j=1,2,\dots,k\right\}.$$ Here $s$ is defined as follows: if we let $\lambda>0$ denote the smallest eigenvalue in the discrete spectrum of the Laplacian $\Delta$ on $\Gamma\backslash\mathfrak H$ then $$s=\begin{cases}\frac{1-\sqrt{1-4\lambda}}2, & \text{ if }0<\lambda<\frac14;\\\frac12, &\text{otherwise}.\end{cases}$$
\end{remark}

Now let $\Gamma$ be a subgroup of finite index of $\SL(2,\mathbb Z)$. Then we have the following corollary of the main theorem, which will be explained in section 3.

\begin{corollary}\label{cor}
Let $\Gamma$ be a subgroup of finite index of $\SL(2,\mathbb Z)$. Let $p=\Gamma g\in\Gamma\backslash\SL(2,\mathbb R)$ with $$g=\left(\begin{array}{cc}a & b \\c & d\end{array}\right).$$ If $a/c\in\mathbb R$ is a Diophantine number of type $\zeta$; that is, there exists $C>0$ such that for all $m/n\in\mathbb Q$, we have $$|n|^\zeta\left|n\cdot\left(\frac ac\right)-m\right|\geq C,$$ then the orbit $\{pu(n^{1+\gamma})|n\in\mathbb N\}$ is equidistributed in $\Gamma\backslash\SL(2,\mathbb R)$ for $0<\gamma<\gamma_0:=\frac{s^2}{(4+s)(\zeta+4)}$.
\end{corollary}

To prove the main theorem, we shall use the technique of Venkatesh in \cite{V} and Str\"ombergsson's result in \cite{S} about effective version of Dani and Smillie's result \cite{DS} on $\Gamma\backslash\SL(2,\mathbb R)$. In fact, an immediate consequence of the technique of \cite{V} and result of \cite{S} is obtained in the following theorem. Before stating the theorem, we need some notations. For $f\in C^k(\Gamma\backslash G)$ we let $\|f\|_{p,k}$ be the Sobolev $L^p$-norm involving all the Lie derivatives of order $\leq k$ of $f$. Note that $\|f\|_{\infty,0}$ is the supremum norm of $f$. We know that $G$ acts on the upper half plane $\mathfrak H$ by the action $$\left(\begin{array}{cc}a & b \\c & d\end{array}\right)\cdot z=\frac{az+b}{cz+d}$$ and we have the standard projection of $\Gamma\backslash G$ to the fundamental domain of $\Gamma$ in $\mathfrak H$ $$\pi:\Gamma\backslash G\to\Gamma\backslash\mathfrak H$$  by sending $\Gamma g$ to $\Gamma g(i)$. We define the geodesic flow on $\Gamma\backslash G$ by $$g_t(\Gamma g)=\Gamma g\left(\begin{array}{cc}e^{t/2} & 0 \\0 & e^{-t/2}\end{array}\right).$$
Fix, once for all, a point $p_0\in\Gamma\backslash\mathfrak H$. For $p\in\Gamma\backslash G$ let $$\dist(p)=d_{\mathfrak H}(p_0,\pi(p))$$ where $d_{\mathfrak H}(\cdot,\cdot)$ is the hyperbolic distance on $\Gamma\backslash\mathfrak H$.

\begin{theorem}[Cf.\cite{V} Theorem 3.1]\label{thm51}
Let $T>K>2$ and $f\in C^\infty(\Gamma\backslash G)$ satisfying $\int_{\Gamma\backslash G}fd\mu=0$ and $\|f\|_{\infty,4}<\infty$. Suppose that $q\in\Gamma\backslash G$ satisfies $r=r(q,T)=T\cdot e^{-dist(g_{\log T}(q))}\geq1$. Then we have $$\left|\frac1{T/K}\sum\limits_{\substack{j\in\mathbb Z\\0\leq Kj<T}}f(qu(Kj))\right|\ll\frac{K^\frac12\ln^\frac32(r+2)}{r^\frac\beta2} \|f\|_{\infty,4}$$ for $\beta=\frac{s\kappa}{2(8+\kappa)}$. Here $\kappa$ is the constant in the mixing property of the unipotent flow (see Theorem \ref{thm:KM} and Remark \ref{remark1}) and $s$ is defined as in Remark \ref{remark}.
\end{theorem}

This theorem gives an estimate for the average of the unipotent action along an arithmetic progression with gap $K$, which is crucial in our proof of the main theorem. This was proved first in \cite{V} and later in \cite{TV}, both in the case of $\Gamma\backslash G$ being compact.\\

The strategy of the paper is the following: note that the bound in Theorem \ref{thm51} depends on the initial point, and hence when we combine the results with different arithmetic progressions and different initial points, the outcomes would get out of control. To overcome this difficulty, we need the Diophantine condition. With the help of this Diophantine condition along with the notion of $(C,\alpha;\rho,\epsilon_0)$-good functions, we will be able to control the rates of these effective results. In section 2, we list the concepts and theorems that we need in this paper. In section 3, we study the Diophantine condition and deduce Corollary \ref{cor} from the main theorem. In section 4, we will study dynamics of a special class of orbits in $\Gamma\backslash G$. The dynamical properties of these orbits will help us control the rates of the 
effective results in this paper. Since we are dealing with the noncompact case of $\Gamma\backslash G$, and also for the sake of completeness, we include the technique of \cite{V} and prove Theorem \ref{thm51} in section 5. We will finish the proof of the main theorem in section 6. Further discussions will be included in section 7.\\

It may be interesting to explore the relation between the techniques used in this work and those developed in the work of Sarnak and Ubis \cite{SU}, where they have described the limiting distribution of horocycles at primes.\\

\noindent$\textbf{Acknowledgement.}$ I would like to thank Professor Nimish Shah for his helps and encouragement. Many ideas in this paper come from discussions with him. I would like to thank a referee for suggesting that the Diophantine condition is related to geodesic excursion rates (Lemma \ref{l33}) and that Theorem 7.1 would follow from \cite{MP}. I thank referees for careful reading and suggestions on improvement of proofs at various places. I am grateful to Nicolas de Saxce who also independently suggested Lemma \ref{l33}. I thank the MSRI for hospitality during spring 2015.

\section{Prerequisites}
Throughout this note, if there exists an absolute constant $C>0$ such that $f\leq Cg$, then we write $f\ll g$. If $f\ll g$ and $g\ll f$, then we use the notation $f\sim g$.  We denote $G=\SL(2,\mathbb R)$ and $\Gamma$ a non-uniform lattice in $G$. Let $$N=\{u(t)|t\in\mathbb R\},\quad A=\left\{\left(\begin{array}{cc}s & 0 \\0 & s^{-1}\end{array}\right)\Big| s\in{\mathbb R}_{+}\right\}.$$ For any element $a\in A$, we denote $\alpha(a)=s$.\\ 

One of the ingredients in our calculations is the effective version of the mixing property of unipotent flows in $\Gamma\backslash G$. The following effective version is proved by Kleinbock and Margulis \cite{KM1}. 

\begin{theorem}[Kleinbock and Margulis \cite{KM1}]\label{thm:KM}
There exists $\kappa>0$ such that for any $f,g\in C^\infty (\Gamma\backslash G)$, we have $$\left|(u(t)\cdot f,g)-\int_{\Gamma\backslash G}f\int_{\Gamma\backslash G}g\right|\ll (1+|t|)^{-\kappa}\|f\|_{\infty,1}\|g\|_{\infty,1}.$$ Here $(u(t)\cdot f)(x)=f(xu(t))$ is the right translation of $f$ by $u(t)$.
\end{theorem}
\begin{remark}\label{remark1}
Note that when $G=\SL(2,\mathbb R)$, we can calculate $\kappa$ explicitly. Indeed, let $\lambda>0$ denote the smallest eigenvalue in the discrete spectrum of the Laplacian $\Delta$ on $\Gamma\backslash\mathfrak H$, then it follows from \cite{V} formula (9.7) and the technique of Lemma 2.3 in \cite{R} that $\kappa=2s-\epsilon$ for any $\epsilon>0$. Here $s$ is defined as in Remark \ref{remark}.
\end{remark}

Another ingredient in the calculations is the effective version of Dani and Smillie's result \cite{DS} proved by Str\"ombergsson \cite{S}.

\begin{theorem}[Str\"ombergsson \cite{S}]\label{thm:S}
For all $p\in\Gamma\backslash G$, $T\geq10$, and all $f\in C^4(\Gamma\backslash G)$ such that $\|f\|_{\infty,4}<\infty$ we have
\begin{eqnarray*}
\left|\frac1T\int_0^Tf(pu(t))dt-\int_{\Gamma\backslash G}fd\mu\right|\leq
O(\|f\|_{\infty,4})r^{-s}\ln^3(r+2)
\end{eqnarray*}
provided that $r\geq1$. Here $s>0$ is a number depending on the spectrum of the Laplacian on $\Gamma\backslash\mathfrak H$ and $r=r(p,T)=T\cdot e^{-dist(g_{\log T}(p))}$. The implied constants depend only on $\Gamma$ and $p_0$.
\end{theorem}
\begin{remark}
Here we can take $s$ as in Remark \ref{remark}, i.e., let $\lambda>0$ be the smallest eigenvalue in the discrete spectrum of the Laplacian $\Delta$ on $\Gamma\backslash\mathfrak H$, and $$s=\begin{cases}\frac{1-\sqrt{1-4\lambda}}2, & \text{ if }0<\lambda<\frac14;\\\frac12, &\text{otherwise}.\end{cases}$$ Readers may refer to \cite{S} for more details. We will prove a weaker version of this theorem in Appendix B using only mixing property, and Theorem \ref{main theorem} could be proved by using techniques only from dynamical systems for many Diophantine points (at least for a subset of full Haar measure).
\end{remark}

\section{The Diophantine Condition}
First we deduce Corollary \ref{cor} from the main theorem.
\begin{proof}[Proof of Corollary \ref{cor}]
If $\Gamma$ is a subgroup of finite index of $\SL(2,\mathbb Z)$, then we can pick $\sigma_j\in\SL(2,\mathbb Z)$ $(1\leq j\leq k)$. Now let $p=\Gamma g\in\Gamma\backslash G$ with $$g=\left(\begin{array}{cc}a & b \\c & d\end{array}\right).$$
Note that for each $m_j$, we have $$m_j(\pi_j^{-1}(p))\subseteq g^{-1}\mathbb Z^2\setminus\{0\}=\left\{\left(\begin{array}{c}dm-bn  \\-cm+an \end{array}\right)\Big| (m,n)\in\mathbb Z^2\setminus\{0\}\right\}.$$
If there exist constants $\zeta>0,\mu,\nu>0$ such that for any $(m,n)\in\mathbb Z^2\setminus\{0\}$
\begin{equation}\label{c}
|an-cm|\geq\mu\text{\; or\; }|dm-bn|^\zeta |an-cm|\geq\nu,
\end{equation}
 then $p$ is Diophantine of type $(\zeta,\dots,\zeta)$ by the definition above. In particular, if $a/c\in\mathbb R$ is a Diophantine number of type $\zeta$, i.e. there exists $C>0$ such that for $m/n\in\mathbb Q$, $$|n|^\zeta\left|n\cdot\frac ac-m\right|\geq C,$$ then condition $(\ref{c})$ holds because when $|an-cm|$ is sufficiently small, $$|dm-bn|=\frac{|cdm-bcn|}c=\frac{|cdm-(ad-1)n|}c=\frac{|d(cm-an)+n|}c\sim |n|.$$ Hence, Corollary \ref{cor} follows from the main theorem.
\end{proof}

In order to prove the main theorem, we have to analyze the map $m_j:\Gamma_j\backslash G\to\mathbb R^2/\pm$ for each $j$. The following lemma is well known. The reader may refer to \cite{DS}.  We will denote $B_d$ the ball of radius $d$ around the origin in $\mathbb R^2$.

\begin{lemma}[\cite{DS} Lemma 2.2]\label{l31}
For each $j$ with the maps $\pi_j:\Gamma_j\backslash G\to\Gamma\backslash G$ and $m_j:\Gamma_j\backslash G\to\mathbb R^2/\pm$, there exists a constant $d_j>0$ such that for any $p\in\Gamma\backslash G$ there exists at most one point of $m_j(\pi_j^{-1}(p))$ which lies in $B_{d_j}$.
\end{lemma}
\begin{remark} \label{rem:dj}
We will fix these $d_j$'s for $j=1,2,\dots, k$ throughout this note.
\end{remark}

\begin{lemma}\label{l}
If $p\in\Gamma\backslash G$ is Diophantine of type $(\kappa_1,\kappa_2,\dots,\kappa_k)$, then the orbit $\{g_t(p)|t\geq0\}$ is non-divergent.
\end{lemma}
\begin{proof}
Suppose that $\{g_t(p)|t\geq0\}$ is divergent. Let $\eta_j$ be the cusp where $\{g_t(p)|t\geq0\}$ diverges. By Lemma 11.29 in \cite{EW}, we know that $\{g_t(p)|t\geq0\}$ is divergent if and only if $\{pu(t)\}$ is periodic in $\Gamma\backslash G$. Combined with Lemma 2.1 in \cite{DS}, this would imply that there is a point $\left(\begin{array}{c}x \\y\end{array}\right)\in m_j(\pi_j^{-1}(p))$ lying on the $x$-axis in $\mathbb R^2$, i.e. $y=0$, which contradicts the Diophantine condition. Therefore, $\{g_t(p)|t\geq0\}$ is non-divergent.
\end{proof}

\begin{definition}
For $p\in\Gamma\backslash G$, we define $$\|p\|_j:=\min\left\{\left\|\left(\begin{array}{c}a \\b\end{array}\right)\right\|\Big|\left(\begin{array}{c}a \\b\end{array}\right)\in m_j(\pi_j^{-1}(p))\right\}$$ where $\|\cdot\|$ denotes the standard Euclidean norm in $\mathbb R^2$. Moreover, we define $$d(p)=\min\{\|p\|_j|j=1,2,\dots,k\}.$$ 
\end{definition}

\begin{lemma}\label{l32}
For any $p\in\Gamma\backslash SL(2,\mathbb R)$, we have $$e^{\dist(p)}\sim\frac1{d(p)^2}.$$
\end{lemma}
\begin{proof}
Recall that $\eta_j$ $(1\leq j\leq k)$ are the inequivalent cusps of $\Gamma\backslash\mathfrak H$. For each $1\leq j\leq k$, we fix a small neighborhood $C_j$ of $\eta_j$ in $\Gamma\backslash G$ such that $C_1,C_2,\dots, C_k$ are pairwise disjoint. Also we fix a point $q_j\in C_j$ for each $1\leq j\leq k$.
We observe that it suffices to prove the lemma for $p\in C_j$ $(j=1,2,\dots,k)$ since the complement of $\bigcup C_j$ is compact. 
Let $p\in C_j$ for some $j\in\{1,2,\dots,k\}$. Let $\alpha_j>0$ be such that $\pi_j$ maps $\sigma_jN_{\Omega_j} A_{\alpha_j}K$ isomorphically to $C_j$. 
Then we can pick a representative for $p$ in $\sigma_jN_\Omega A_{\alpha_j}K$, say $\sigma_j n_pa_pk_p$, i.e. $p=\Gamma\sigma_j n_pa_pk_p=\pi_j(\Gamma_j\sigma_j n_pa_pk_p)$. 
By definition we know that $$d(p)=\|p\|_j=\|k_p^{-1}a_p^{-1}e_1\|=\alpha(a_p)^{-1}.$$ 
On the other hand, in the fundamental domain of $\Gamma\backslash\mathfrak H$, the point corresponding to $p=\pi_j(\Gamma_j\sigma_j n_pa_pk_p)\in C_j$ is equal to $$\Gamma\sigma_j n_pa_pk_p\cdot i=\Gamma\sigma_j(n_pa_p\cdot i)=\Gamma\sigma_j(n_p\cdot(\alpha(a_p)^2i)).$$ 
Since $\sigma_j$ is fixed and $n_p$ is in the compact set $N_{\Omega_j}$ of $N$, we obtain $$|d_{\mathfrak H}(\pi(q_j),\pi(p))-\ln{\alpha(a_p)^2}|\leq C_j$$ for some constant $C_j>0$. Since $q_j$ is fixed, we have $$|\dist(p)-d_{\mathfrak H}(\pi(q_j),\pi(p))|\leq C_j'$$ for some $C_j'>0$. Therefore we get $$|\dist(p)-\ln{\alpha(a_p)^2}|\leq C$$ for $C=\max\{C_1+C_1',C_2+C_2',\dots,C_k+C_k'\}$ and hence $$e^{\dist(p)}=e^{(\dist(p)-\ln{\alpha(a_p)^2})+\ln{\alpha(a_p)^2}}\sim e^{\ln{\alpha(a_p)^2}}=\alpha(a_p)^2=\frac1{d(p)^2}.$$
\end{proof}

Now we prove that the Diophantine condition on $p\in\Gamma\backslash G$ can be defined by the excursion rate of the geodesic orbit $\{g_t(p)\}_{t>0}$. We need some notations. As in the proof of Lemma \ref{l32}, let $\eta_1,\eta_2,\dots,\eta_k$ be the inequivalent cusps of $\Gamma\backslash\mathfrak H$ and we choose the neighborhood $C_j$ of $\eta_j$ $(1\leq j\leq k)$ in $\Gamma\backslash G$ such that $C_1,C_2,\dots,C_k$ are pairwise disjoint. For each $1\leq j\leq k$, we define a function on $\Gamma\backslash G$ by $$\dist^{(j)}(p)=\begin{cases} \dist(p), & \text{ if }p\in C_j;\\ 0, &\text{ otherwise.}
\end{cases}$$

\begin{lemma}\label{l33}
A point $p\in\Gamma\backslash G$ has Diophantine type $(\kappa_1,\kappa_2,\dots,\kappa_k)$ if and only if $\kappa_j\geq1$ and $$\limsup_{t\to\infty}\left(\dist^{(j)}(g_t(p))-\frac{\kappa_j-1}{\kappa_j+1}t\right)<\infty$$ for each $j\in\{1,2,\dots,k\}$.
\end{lemma}
\begin{proof}
As in the proof of Lemma \ref{l32}, we know that there exists $\alpha_j>0$ such that $\pi_j$ maps $\sigma_jN_{\Omega_j}A_{\alpha_j}K$ isomorphically to $C_j$. Also using the same argument in the proof of Lemma \ref{l32}, we can get that for any $q\in C_j$, 

\begin{eqnarray}\label{eqj}
e^{\dist^{(j)}(q)}\sim\frac1{\|q\|^2_j}.
\end{eqnarray}

If $p\in\Gamma\backslash G$ is Diophantine of type $(\kappa_1,\kappa_2,\dots,\kappa_k)$, then for each $j$,  any $\left(\begin{array}{c}a \\b\end{array}\right)\in m_j(\pi_j^{-1}(p))$ satisfies $$|b|\geq\mu_j\text{ or } |a|^{\kappa_j}|b|\geq\nu_j.$$ Since 
\begin{eqnarray}\label{id}
m_j(\pi_j^{-1}(g_t(p)))=\left(\begin{array}{cc}e^{-t/2} & 0 \\0 & e^{t/2}\end{array}\right)m_j(\pi_j^{-1}(p)),
\end{eqnarray}
this implies that any $\left(\begin{array}{c}x \\y\end{array}\right)\in m_j(\pi_j^{-1}(g_t(p)))$ satisfies 
\begin{eqnarray}\label{eq0}
|y|\geq e^{t/2}\mu_j\text{ or }|x|^{\kappa_j}|y|\geq\nu_j e^{(1-\kappa_j)t/2}.
\end{eqnarray}
Note that this holds for all $t>0$. 

By Lemma \ref{l}, we know that $\{g_t(p)|t\geq0\}$ is nondivergent. So there exists a compact subset $S\subset\Gamma\backslash G$ such that $g_{t_i}(p)$ remains in $S$ for infinitely many $t_i\to\infty$. By the compactness of $S$, we can find a constant $C_0>0$ such that for each $t_i$, there exists $\left(\begin{array}{c}x_i \\y_i\end{array}\right)\in m_j(\pi_j^{-1}(g_{t_i}(p)))$ satisfying $$\|(x_i,y_i)\|\leq C_0.$$
This implies via equation (\ref{eq0}) that $\kappa_j\geq1$. Now fix $t$ and for any$\left(\begin{array}{c}x \\y\end{array}\right)\in m_j(\pi_j^{-1}(g_t(p)))$ we have $$\|(x,y)\|\geq|y|\geq e^{t/2}\mu_j\text{ or }\|(x,y)\|\geq\max\{|x|,|y|\}\geq|x|^{\frac{\kappa_j}{\kappa_j+1}}|y|^{\frac1{\kappa_j+1}}\geq\nu_j^{\frac1{\kappa_j+1}}e^{\frac{1-\kappa_j}{1+\kappa_j}\frac t2}.$$ Therefore, $\|g_t(p)\|_j\geq e^{t/2}\mu_j$ or $\nu_j^{\frac1{\kappa_j+1}}e^{\frac{1-\kappa_j}{1+\kappa_j}t/2}$. By equation (\ref{eqj}) and $\kappa_j\geq1$, we get that
$$\limsup_{t\to\infty}\left(\dist^{(j)}(g_t(p))-\frac{\kappa_j-1}{\kappa_j+1}t\right)<\infty.$$

Conversely, if the above inequality holds for each $j$ with $\kappa_j\geq1$,  then by equation (\ref{eqj}) there exists a constant $C>0$ such that for all $t>0$ we have $$\|g_t(p)\|_j\geq Ce^{\frac{1-\kappa_j}{1+\kappa_j}\frac t2}.$$ This implies via equation (\ref{id}) that for any $\left(\begin{array}{c}a \\b\end{array}\right)\in m_j(\pi_j^{-1}(p))$ we have
\begin{eqnarray}\label{jeq}
e^{-t}a^2+e^tb^2\geq Ce^{\frac{1-\kappa_j}{1+\kappa_j}t}.
\end{eqnarray}
By discreteness of $m_j(\pi_j^{-1}(p))$ in $\mathbb R^2$, there exists a constant $\mu_j>0$ such that if $\left(\begin{array}{c}a \\b\end{array}\right)\in m_j(\pi_j^{-1}(p))$ satisfies $|b|<\mu_j$, then $|b|<|a|$. Now for such $\left(\begin{array}{c}a \\b\end{array}\right)$, we take $t>0$ such that $e^{-t}a^2=e^tb^2$. By equation (\ref{jeq}), this implies that $|b|\geq \sqrt{\frac C2}e^{-\frac{\kappa_j}{1+\kappa_j}t}$ and hence $$|a|^{\kappa_j}|b|=|e^tb|^{\kappa_j}|b|\geq\left(\sqrt{\frac C2}\right)^{\kappa_j+1}.$$ This implies that $p$ is Diophantine of type $(\kappa_1,\kappa_2,\dots,\kappa_k)$.
\end{proof}

\section{$(C,\alpha;\rho,\epsilon_0)$-good functions in presence of the Diophantine condition }
This section will be important in the proof of the main theorem. First, we need a modified version of the concept of $(C,\alpha)$-good functions (see \cite{KM2} for the definition of $(C,\alpha)$-good functions).

\begin{definition}
A function $f(x)$ is said to be $(C,\alpha;\rho,\epsilon_0)$-good if for any $0<\epsilon<\epsilon_0$ and any $I=(x_1,x_2)\subset [1,\infty)$ with $|f(x_1)|=\rho$, we have
$$m(\{x\in I||f(x)|\leq\epsilon\})\leq C\left(\frac\epsilon\rho\right)^\alpha m(I)$$
where $m$ denotes the Lebesgue measure on $\mathbb R$.
\end{definition}

Now we shall begin to study a special class of functions and prove that they are $(C,\alpha;\rho,\epsilon_0)$-good for some $C,\alpha,\rho$ and $\epsilon_0>0$. Note that we restrict these functions to the domain $[1,\infty)$.\\

\begin{lemma}\label{l41}
Let $\kappa,\mu,\nu>0$ and $0<\gamma<\frac1{\kappa+4}$. Let $\left(\begin{array}{c}a \\b\end{array}\right)\in\mathbb R^2\setminus\{0\}$ be such that $$|b|\geq\mu\text{\; or\; }|a|^\kappa|b|\geq\nu.$$Then there exist $C,\epsilon_0>0$ such that $$f(x)=(bx^{\frac34+\gamma}-ax^{-\frac14})^2(x^{\frac14-\frac1{\kappa+4}})^2+(bx^{\frac14})^2(x^{\frac14-\frac1{\kappa+4}})^2$$ is $(C,\frac12;\rho,\epsilon_0)$-good on $[1,\infty)$, where $\rho$ is any fixed constant $\leq f(1)$. Here the constants $C,\epsilon_0$ depend only on $\rho, \kappa,\mu,\nu$ and $\gamma$.
\end{lemma}
\begin{proof}
We observe that if $|b|\geq\mu$, then $f(x)\geq(bx^{\frac14})^2(x^{\frac14-\frac1{\kappa+4}})^2\geq b^2\geq\mu^2$ and $f(x)$ is automatically $(C,\alpha;\rho,\epsilon_0)$-good for any $C,\alpha,\rho$ and $\epsilon_0=\mu^2/2$. Therefore, in the following we assume that $|b|<\mu$ and hence $|a|^\kappa|b|\geq\nu$. We have two cases: $ab<0$ and $ab>0$.\\

\noindent Case 1: $ab<0.$ Our function $f(x)$ then becomes $$f(x)=(|b|x^{\frac34+\gamma}+|a|x^{-\frac14})^2(x^{\frac14-\frac1{\kappa+4}})^2+(bx^{\frac14})^2(x^{\frac14-\frac1{\kappa+4}})^2.$$ We have
\begin{eqnarray*}
f(x)&\geq&\left((|b|x^{\frac34+\gamma}+|a|x^{-\frac14})x^{\frac14-\frac1{\kappa+4}}\right)^2\\
&\geq&\left(\max\{|b|x^{1+\gamma-\frac1{\kappa+4}},|a|x^{-\frac1{\kappa+4}}\}\right)^2\\
&\geq&\left((|b|x^{1+\gamma-\frac1{\kappa+4}})^\frac1{\kappa+1}(|a|x^{-\frac1{\kappa+4}})^{\frac\kappa{\kappa+1}}\right)^2\\
&\geq&\left((|b||a|^\kappa x^{1+\gamma-\frac{\kappa+1}{\kappa+4}})^\frac1{\kappa+1}\right)^2\\
&\geq&\nu^\frac2{\kappa+1}.
\end{eqnarray*}
This implies that $f(x)$ is $(C_2,\alpha_2;\rho,\epsilon_0)$-good for any $C_2,\alpha_2>0$ and $\epsilon_0=\frac12\nu^{\frac2{\kappa+1}}$.\\

\noindent Case 2: $ab>0$. Without loss of generality, we assume that $a>0,b>0$. Now let $I=(x_1,x_2)\subset[1,\infty)$ be an interval $(x_1,x_2)$ where $f(x_1)=\rho$. Since $f(x_1)=\rho$, we know that $$\text{either\, }(bx_1^{\frac34+\gamma}-ax_1^{-\frac14})^2(x_1^{\frac14-\frac1{\kappa+4}})^2\geq\frac{\rho}2\text{\, or\, }(bx_1^{\frac14})^2(x_1^{\frac14-\frac1{\kappa+4}})^2\geq\frac{\rho}2.$$ If $(bx_1^{\frac14})^2(x_1^{\frac14-\frac1{\kappa+4}})^2\geq\rho/2$, then there is nothing to prove because $f(x)\geq\rho/2$ for all $x\geq x_1$. 

Otherwise, we have $$(bx_1^{\frac34+\gamma}-ax_1^{-\frac14})^2(x_1^{\frac14-\frac1{\kappa+4}})^2\geq\frac{\rho}2.$$ Note that $$\{x\in I| |f(x)|\leq\epsilon\}\subseteq\{x\in I|(bx^{\frac34+\gamma}-ax^{-\frac14})^2(x^{\frac14-\frac1{\kappa+4}})^2\leq\epsilon\}.$$ Therefore, to finish the proof of the lemma, it suffices to show that there exist $C,\epsilon_0>0$ depending only on $\rho,\kappa,\mu,\nu,\gamma$ such that for any $0<\epsilon<\epsilon_0$ we have
\begin{eqnarray}\label{eq1}
\frac1{x_2-x_1}m(\{x\in(x_1,x_2)\big| |g(x)|\leq\sqrt\epsilon\})\leq C\left(\frac\epsilon{\rho}\right)^\frac12
\end{eqnarray}
where $g(x)=(bx^{\frac34+\gamma}-ax^{-\frac14})x^{\frac14-\frac1{\kappa+4}}=bx^{1+\gamma-\frac1{\kappa+4}}-ax^{-\frac1{\kappa+4}}$. Note that $g(x)$ is increasing and $|g(x_1)|\geq\sqrt{\rho/2}$. Without loss of generality, we may assume that $|g(x_1)|=\sqrt{\rho/2}$. If $g(x_1)=\sqrt{\rho/2}$, since $g(x)$ is increasing, the $(C,\alpha;\rho,\epsilon_0)$-good property automatically holds in this case with $\epsilon_0=\frac12\sqrt{\rho/2}$. Therefore we assume that $g(x_1)=-\sqrt{\rho/2}$. In this case, we will prove that the inequality (\ref{eq1}) holds with $\epsilon_0=\frac12\sqrt{\rho/2}$. 

Let $0<\epsilon<\epsilon_0$. Since $g(x)$ is increasing with $g(x_1)=-\sqrt{\rho/2}$, if we fix $x_1$ and let $x_2$ vary as a variable, then the ratio $$\frac1{x_2-x_1}m(\{x\in(x_1,x_2)\big| |g(x)|\leq\sqrt\epsilon\})$$ would attain its maximum when $g(x_2)=\sqrt\epsilon$. So we will assume that $g(x_2)=\sqrt\epsilon$. To compute this maximal ratio, let $z\in (x_1,x_2)$ such that $g(z)=-\sqrt\epsilon$, and then by the mean value theorem we obtain
\begin{eqnarray}\label{eq}\nonumber
&{}&\frac1{x_2-x_1}m(\{x\in(x_1,x_2)\big| |g(x)|\leq\sqrt\epsilon\})\\ \nonumber
&=&\frac{x_2-z}{x_2-x_1}=\frac{g'(\xi_2)}{g'(\xi_1)}\cdot\frac{g(x_2)-g(z)}{g(x_2)-g(x_1)}\\
&=&\frac{g'(\xi_2)}{g'(\xi_1)}\cdot\frac{2\sqrt\epsilon}{\sqrt\epsilon+\sqrt{\rho/2}}
\end{eqnarray}
where $\xi_1$ is between $x_2$ and $z$, $\xi_2$ is between $x_1$ and $x_2$.

Let $x_3\in[1,\infty)$ such that $g(x_3)=\sqrt{\rho/2}$. Then $(x_1,x_2)\subset(x_1,x_3)$. According to equation (\ref{eq}), to prove formula (\ref{eq1}), it suffices to prove that for any $x,y\in (x_1,x_3)$ the ratio $$\frac{g'(x)}{g'(y)}$$ is bounded above by constants depending only on $\rho,\kappa,\mu,\nu$ and $\gamma$. Observe that $$g'(x)=\left(1+\gamma-\frac1{\kappa+4}\right)bx^{\gamma-\frac1{\kappa+4}}+\frac{a}{\kappa+4}x^{-\frac{\kappa+5}{\kappa+4}}$$ is decreasing since $\gamma<\frac1{\kappa+4}$. Therefore, to get an upper bound for $g'(x)/g'(y)$ $(x,y\in(x_1,x_3))$, we only need to estimate $g'(x_1)/g'(x_3)$. By the condition that $g(x_1)=-\sqrt{\rho/2}$ and $g(x_3)=\sqrt{\rho/2}$, we have
\begin{eqnarray*}
\frac{g'(x_1)}{g'(x_3)}&=&\frac{\left(1+\gamma-\frac1{\kappa+4}\right)bx_1^{\gamma-\frac1{\kappa+4}}+\frac{a}{\kappa+4}x_1^{-\frac{\kappa+5}{\kappa+4}}}{\left(1+\gamma-\frac1{\kappa+4}\right)bx_3^{\gamma-\frac1{\kappa+4}}+\frac{a}{\kappa+4}x_3^{-\frac{\kappa+5}{\kappa+4}}}\\
&=&\frac{\left(1+\gamma-\frac1{\kappa+4}\right)(ax_1^{-\frac1{\kappa+4}}-\sqrt{\rho/2})/x_1+\frac{a}{\kappa+4}x_1^{-\frac{\kappa+5}{\kappa+4}}}{\left(1+\gamma-\frac1{\kappa+4}\right)(ax_3^{-\frac1{\kappa+4}}+\sqrt{\rho/2})/x_3+\frac{a}{\kappa+4}x_3^{-\frac{\kappa+5}{\kappa+4}}}\\
&\leq&\frac{\left(1+\gamma-\frac1{\kappa+4}\right)ax_1^{-\frac1{\kappa+4}}/x_1+\frac{a}{\kappa+4}x_1^{-\frac{\kappa+5}{\kappa+4}}}{\left(1+\gamma-\frac1{\kappa+4}\right)ax_3^{-\frac1{\kappa+4}}/x_3+\frac{a}{\kappa+4}x_3^{-\frac{\kappa+5}{\kappa+4}}}=\left(\frac{x_3}{x_1}\right)^{\frac{\kappa+5}{\kappa+4}}.
\end{eqnarray*}
Now let $x_0\in(x_1,x_3)$ such that $g(x_0)=0$. ($x_0=(a/b)^{\frac1{1+\gamma}}$ by solving the equation $g(x)=0$). We set $x_1=\delta_1x_0$ and $x_3=\delta_2x_0$ for some $\delta_1,\delta_2$. Then $\delta_1,\delta_2$ satisfy the following equation 
$$|b(x_0\delta)^{1+\gamma-\frac1{\kappa+4}}-a(x_0\delta)^{-\frac1{\kappa+4}}|=\sqrt{\frac\rho2}$$ since $|g(x_1)|=|g(x_3)|=\sqrt{\rho/2}$. By the fact that $bx_0^{1+\gamma-\frac1{\kappa+4}}=ax_0^{-\frac1{\kappa+4}}$ and $ba^\kappa\geq\nu$, this equation becomes $$|ax_0^{-\frac1{\kappa+4}}\delta^{1+\gamma-\frac1{\kappa+4}}-ax_0^{-\frac1{\kappa+4}}\delta^{-\frac1{\kappa+4}}|=\sqrt{\frac\rho2}$$ 
\begin{eqnarray*}
|\delta^{1+\gamma-\frac1{\kappa+4}}-\delta^{-\frac1{\kappa+4}}|&=&\sqrt{\frac\rho2}\frac{x_0^\frac1{\kappa+4}}a=\sqrt{\frac\rho2}\frac{(a^{\kappa+1}/ba^\kappa)^{\frac1{(1+\gamma)(\kappa+4)}}}a\\
&\leq&\sqrt{\frac\rho2}\frac1{\nu^{\frac1{(1+\gamma)(\kappa+4)}}}\frac{a^{\frac{\kappa+1}{(1+\gamma)(\kappa+4)}}}a.
\end{eqnarray*}
Here $\frac{\kappa+1}{(1+\gamma)(\kappa+4)}<1$. Since $b\leq\mu$ and hence $a\geq\sqrt[\kappa]{\nu/\mu}$, the above inequality becomes 
$$|\delta^{1+\gamma-\frac1{\kappa+4}}-\delta^{-\frac1{\kappa+4}}|\leq\sqrt{\frac\rho2}\frac1{\nu^{\frac1{(1+\gamma)(\kappa+4)}}}\left(\sqrt[\kappa]{\frac\nu\mu}\right)^{\frac{\kappa+1}{(1+\gamma)(\kappa+4)}-1}$$
which holds for $\delta=\delta_1$ and $\delta_2$. This shows that $\delta_1,\delta_2$ are bounded above and below by constants depending only on $\rho,\kappa,\mu,\nu$ and $\gamma$. Therefore $$\frac{g'(x_1)}{g'(x_3)}\leq \left(\frac{x_3}{x_1}\right)^{(\kappa+5)/(\kappa+4)}=\left(\frac{\delta_2}{\delta_1}\right)^{(\kappa+5)/(\kappa+4)}$$ is bounded above by a constant depending only on $\rho,\kappa,\mu,\nu$ and $\gamma$. This completes the proof of the lemma.\\
\end{proof}

For the rest of this section, we turn to the dynamics on $\Gamma\backslash G$. For later use, we give the following definition.

\begin{definition} \label{def:S}
For any $\delta>0$ and any $j\in\{1,2,\dots,k\}$, we define the subset of $\Gamma\backslash G$ $$S_{j,\delta}:=\{q\in\Gamma\backslash G\big|\|q\|_j\leq\delta\}.$$ Moreover, we define $$S_\delta:=\bigcup\limits_jS_{j,\delta}=\{q\in\Gamma\backslash G\big|d(q)\leq\delta\}.$$
\end{definition}

\begin{lemma}\label{l42}
Let $p\in\Gamma\backslash G$ be Diophantine of type $(\kappa_1,\kappa_2,\dots,\kappa_k)$. We fix $j\in\{1,2,\dots,k\}$ and let $0<\gamma<1/(\kappa_j+4)$. Then for sufficiently small $\epsilon>0$ and $T\geq1$, we have $$\frac1T m\left(\left\{x\in[1,T]\Big|p\left(\begin{array}{cc}x^{\frac14} & x^{\frac34+\gamma} \\0 & x^{-\frac14}\end{array}\right)\in S_{j,\epsilon x^{-\frac14+\frac1{\kappa_j+4}}}\right\}\right)\leq C\epsilon$$ where $C$ is a constant only depending on $p$ and $\gamma$.
\end{lemma}
\begin{proof}
We will use the notations, the maps $m_j$ and $\pi_j$ in section 3. Then the image of $\pi_j^{-1}\left(p\left(\begin{array}{cc}x^{\frac14} & x^{\frac34+\gamma} \\0 & x^{-\frac14}\end{array}\right)\right)$ under $m_j$ is equal to $$\left(\begin{array}{cc}x^{\frac14} & x^{\frac34+\gamma} \\0 & x^{-\frac14}\end{array}\right)^{-1}m_j(\pi_j^{-1}(p))=\left\{\left(\begin{array}{cc}x^{-\frac14} & -x^{\frac34+\gamma} \\0 & x^{\frac14}\end{array}\right)\left(\begin{array}{c}a \\ b\end{array}\right)\right\}=\left\{\left(\begin{array}{c}ax^{-\frac14}-bx^{\frac34+\gamma} \\bx^{\frac14}\end{array}\right)\right\}$$ where $\left(\begin{array}{c}a \\b\end{array}\right)$ runs over all points in $m_j(\pi_j^{-1}(p))$. By definition, what we need to prove is equivalent to the following 
\begin{eqnarray*}
m\left(\left\{x\in[1,T]\Big|\left\|p\left(\begin{array}{cc}x^{\frac14} & x^{\frac34+\gamma} \\0 & x^{-\frac14}\end{array}\right)\right\|_j\leq\epsilon x^{-\frac14+\frac1{\kappa_j+4}}\right\}\right)\leq C\epsilon T.
\end{eqnarray*}
 which is equivalent to the following 
\begin{eqnarray*}
&{}&m\left(\left\{x\in[1,T]\Big|\exists\text{ a point in }m_j\left(\pi_j^{-1}\left(p\left(\begin{array}{cc}x^{\frac14} & x^{\frac34+\gamma} \\0 & x^{-\frac14}\end{array}\right)\right)\right)\text{with length}\leq\epsilon x^{-\frac14+\frac1{\kappa_j+4}}\right\}\right)\\
&\leq& C\epsilon T.
\end{eqnarray*}
Let $\rho=\min\left\{d\left(p\left(\begin{array}{cc}1 & 1 \\0 & 1\end{array}\right)\right),d_1,d_2,\dots,d_k\right\}$, where $d_{j}$'s are as in Remark~\ref{rem:dj}. 
 We denote by $P$ the subset $m_j(\pi_j^{-1}(p))$. For $(a,b)\in P$, let $I_{(a,b)}^l (l=1,2,\dots)$ be all the maximal connected subintervals in $[1,T]$ such that for any $x\in I_{(a,b)}^l$ the point of $\pi_j^{-1}\left(p\left(\begin{array}{cc}x^{\frac14} & x^{\frac34+\gamma} \\0 & x^{-\frac14}\end{array}\right)\right)$ corresponding to $(a,b)$; that is $$\left(\begin{array}{cc}x^{\frac14} & x^{\frac34+\gamma} \\0 & x^{-\frac14}\end{array}\right)^{-1}\left(\begin{array}{c}a \\ b\end{array}\right)=\left(\begin{array}{c}ax^{-\frac14}-bx^{\frac34+\gamma} \\bx^{\frac14}\end{array}\right)$$ has norm $\leq\rho x^{-\frac14+\frac1{\kappa_j+4}}$. Since $x\geq1$, Lemma \ref{l31} implies that all the intervals $$\{I_{(a,b)}^l|(a,b)\in P,\;l=1,2,\dots\}$$ are pairwise disjoint. Therefore, we have
\begin{eqnarray*}
&{}&m\left(\left\{x\in[1,T]\Big|\exists\text{a point in }m_j\left(\pi_j^{-1}\left(p\left(\begin{array}{cc}x^{\frac14} & x^{\frac34+\gamma} \\0 & x^{-\frac14}\end{array}\right)\right)\right)\text{with length}\leq\epsilon x^{-\frac14+\frac1{\kappa_j+4}}\right\}\right)\\
&=&\sum\limits_{(a,b)\in P}\sum\limits_lm\left(x\in I_{(a,b)}^l\Big|\left\|\left(\begin{array}{c}ax^{-\frac14}-bx^{\frac34+\gamma} \\bx^{\frac14}\end{array}\right)\right\|\leq\epsilon x^{-\frac14+\frac1{\kappa_j+4}}\right)
\end{eqnarray*}
Because of this, to prove the lemma, it suffices to prove the following $$m\left(\left\{x\in I^l_{(a,b)}\Big|\left\|\left(\begin{array}{c}ax^{-\frac14}-bx^{\frac34+\gamma} \\bx^{\frac14}\end{array}\right)\right\|\leq\epsilon x^{-\frac14+\frac1{\kappa_j+4}}\right\}\right)\leq C\epsilon m(I^l_{(a,b)}),$$  for some $C,\epsilon_0$ with $0<\epsilon<\epsilon_0$, or to prove that the function $$f(x)=(bx^{\frac34+\gamma}-ax^{-\frac14})^2(x^{\frac14-\frac1{\kappa_j+4}})^2+(bx^{\frac14})^2(x^{\frac14-\frac1{\kappa_j+4}})^2$$ has $(C',1/2;\rho^2,\epsilon_0^2)$-good property for some $C'=C\rho$. This follows immediately from Lemma \ref{l41}.\\
\end{proof}

To conclude this section, we give the following proposition, which is crucial in our proof of the main theorem. It is the discrete version of Lemma \ref{l42}.

\begin{proposition}\label{p41}
Let $p\in\Gamma\backslash G$ be Diophantine of type $(\kappa_1,\kappa_2,\dots,\kappa_k)$. Let $0<\gamma<\min\{1/(\kappa_j+4):j=1,2,\dots,k\}$. Then there exists a constant $C_0>0$ depending only on $p$ and $\gamma$ such that for sufficiently small $\epsilon>0$ and any $N\in\mathbb N$, $$\frac1N \left|\left\{n\in[1,N]\cap\mathbb N\Big|p\left(\begin{array}{cc}n^{\frac14} & n^{\frac34+\gamma} \\0 & n^{-\frac14}\end{array}\right)\in S_{\theta(n,\epsilon)}\right\}\right|\leq C_0\epsilon$$ where 
\begin{equation} \label{eq:theta}
\theta(x,\epsilon)=\epsilon \min\{x^{-\frac14+\frac1{\kappa_j+4}}|j=1,2,\dots,k\}.
\end{equation}
\end{proposition}
\begin{proof} By the definition of $S_\delta$, it suffices to prove that there exists a constant $C_0>0$ depending only on $p$ and $\gamma$ such that for each $j$ and any $\epsilon>0$, $$\frac1N \left|\left\{n\in[1,N]\cap\mathbb N\Big|p\left(\begin{array}{cc}n^{\frac14} & n^{\frac34+\gamma} \\0 & n^{-\frac14}\end{array}\right)\in S_{j,\epsilon n^{-\frac14+\frac1{\kappa_j+4}}}\right\}\right|\leq C_0\epsilon.$$

We compute that for any $\delta\in(-1,1)$ and $n\geq1$
\begin{eqnarray*}
&{}&\left(\begin{array}{cc}n^{\frac14} & n^{\frac34+\gamma} \\0 & n^{-\frac14}\end{array}\right)^{-1}\left(\begin{array}{cc}(n+\delta)^{\frac14} & (n+\delta)^{\frac34+\gamma} \\0 & (n+\delta)^{-\frac14}\end{array}\right)\\
&=&\left(\begin{array}{cc}n^{-\frac14} & -n^{\frac34+\gamma} \\0 & n^{\frac14}\end{array}\right)\left(\begin{array}{cc}(n+\delta)^{\frac14} & (n+\delta)^{\frac34+\gamma} \\0 & (n+\delta)^{-\frac14}\end{array}\right)\\
&=&\left(\begin{array}{cc}(1+\delta/n)^{\frac14} & n^{-\frac14}(n+\delta)^{\frac34+\gamma}-n^{\frac34+\gamma}(n+\delta)^{-\frac14} \\0 & (1+\delta/n)^{-\frac14}\end{array}\right)\\
&=&\left(\begin{array}{cc}(1+\delta/n)^{\frac14} & ((n+\delta)^{1+\gamma}-n^{1+\gamma})(n(n+\delta))^{-\frac14} \\0 & (1+\delta/n)^{-\frac14}\end{array}\right)
\end{eqnarray*}
which lies in a compact neighborhood $U$ of identity in $\SL(2,\mathbb R)$. Let $$L=\max\{\|g\||g\in U\}$$ where $\|g\|$ denotes the operator norm of $g$ on $\mathbb R^2$. Then by the computations above, we know that
\begin{eqnarray*}
&{}&\frac1N \left|\left\{n\in[1,N]\cap\mathbb N\Big|p\left(\begin{array}{cc}n^{\frac14} & n^{\frac34+\gamma} \\0 & n^{-\frac14}\end{array}\right)\in S_{j,\epsilon n^{-\frac14+\frac1{\kappa_j+4}}}\right\}\right|\\
&\leq&\frac1N m\left(\left\{x\in[1,N]\Big|p\left(\begin{array}{cc}x^{\frac14} & x^{\frac34+\gamma} \\0 & x^{-\frac14}\end{array}\right)\in S_{j,L\epsilon x^{-\frac14+\frac1{\kappa_j+4}}}\right\}\right).
\end{eqnarray*}
Now the proposition follows immediately from Lemma \ref{l42}.\\
\end{proof}

\section{Calculations}
In this section, we shall apply the technique of Venkatesh to obtain some effective results about averaging over arithmetic progressions. It is very similar to \cite{V}, where Venkatesh proved the sparse equidistribution theorem for $\Gamma$ being cocompact. Since in our setting $\Gamma$ is non-uniform, and for the sake of self-containedness, we include the details of the calculations in this section. We will follow the notations in \cite{V}. Throughout this section, we fix an arbitrary point $q\in\Gamma\backslash G$. For a character $\psi:\mathbb R\to S^1$, we define $$\mu_{T,\psi}(f)=\frac1T\int_0^T\psi(t)f(qu(t))dt$$ for $f$ on $\Gamma\backslash G$.

\begin{lemma}[Cf.~{\cite[Lemma 3.1]{V}}] \label{l51}
There exists a constant $\beta>0$ which only depends on $\Gamma$ such that for any $f\in C^\infty(\Gamma\backslash G)$ satisfying $\|f\|_{\infty,4}<\infty$ and $\int_{\Gamma\backslash G}fd\mu=0$, any character $\psi:\mathbb R\to S^1$, any $T\geq1$ and any $q\in\Gamma\backslash G$ satisfying
\begin{equation}\label{eq:r}
r=r(q,T)=T\cdot e^{-\dist(g_{\log T}(q))}\geq1,
\end{equation}
we have $$|\mu_{T,\psi}(f)|\ll r^{-\beta}\ln^3(r+2)\|f\|_{\infty,4}$$
 and the implicit constant is independent of $\psi$.
\end{lemma}
\begin{proof}
The proof is almost the same as that of \cite[Lemma 3.1]{V}  combined with \cite{S}. We define $$\sigma_H(f)(x)=\frac1H\int_0^H\psi(s)f(xu(s))ds.$$ First it is easy to get that $|\mu_{T,\psi}(f)-\mu_{T,\psi}(\sigma_H(f))|\leq\frac HT\|f\|_{\infty,0}\leq\frac Hr\|f\|_{\infty,0}$. Now we estimate $\mu_{T,\psi}(\sigma_H(f))$. By Cauchy-Shwartz inequality, we have
\begin{eqnarray*}
|\mu_{T,\psi}(\sigma_H(f))|&\leq&\frac1T\left(\int_0^T|\psi(t)|^2dt\right)^{\frac12}\left(\int_0^T|\sigma_H(f)(qu(t))|^2dt\right)^{\frac12}\\
&\leq&\left(\frac1T\int_0^T|\sigma_H(f)(qu(t))|^2dt\right)^{\frac12}\\
&=&\left|\frac1{H^2}\int_0^H\int_0^H\psi(z-y)\left(\frac1T\int_0^T\overline{f(qu(t)u(y))}f(qu(t)u(z))dt\right)dydz\right|^\frac12\\
&\leq&\left(\frac1{H^2}\int_0^H\int_0^H\left|\frac1T\int_0^T\overline{f^y}f^z(qu(t))dt\right|dydz\right)^\frac12.
\end{eqnarray*}
Here $f^y$ and $f^z$ denote the right translation of $f$ by $u(y)$ and the right translation of $f$ by $u(z)$, respectively. Therefore, by Strombergsson's effective equidistribution Theorem \ref{thm:S}, we have 
\begin{eqnarray*}
|\mu_{T,\psi}(\sigma_H(f))|&\leq&\left(\frac1{H^2}\int_0^H\int_0^HO(\|\overline{f^y}f^z\|_{\infty,4})r^{-s}\ln^3(r+2)dydz+\frac1{H^2}\int_0^H\int_0^H\left|(f^{y-z},f)\right|dydz\right)^\frac12.
\end{eqnarray*}
for some $s>0$ depending only on $\Gamma$ (the spectral gap) and $r$ is as in (\ref{eq:r}) (see Theorem \ref{thm:S}). 

By mixing property of unipotent flows (Theorem~\ref{thm:KM}), we know that $$|(f^h,f)|\ll (1+|h|)^{-\kappa}\|f\|_{\infty,1}^2.$$ Also by product rule and chain rule in Calculus (see \cite{V} Lemma 2.2 for details), we know that $$O(\|\overline{f^y}f^z\|_{\infty, 4})\ll O(\|\overline{f^y}\|_{\infty,4}\|f^z\|_{\infty,4})\ll y^4z^4O(\|f\|^2_{\infty,4}).$$ Therefore, combining all the computations above, we obtain
\begin{eqnarray*}
|\mu_{T,\psi}(f)|&\leq&|\mu_{T,\psi}(f)-\mu_{T,\psi}(\sigma_H(f))|+|\mu_{T,\psi}(\sigma_H(f))|\\
&\ll&\frac Hr\|f\|_\infty+(H^8 r^{-s}\ln^3(r+2)+H^{-\kappa})^\frac12\|f\|_{\infty,4}.
\end{eqnarray*}
Let $H=r^{\frac s{8+\kappa}}$ and we get $\beta=s\kappa/2(8+\kappa)$. This completes the proof of the lemma.\\
\end{proof}

We deduce Theorem \ref{thm51} from Lemma~\ref{l51}. It will be crucial in the proof of the main theorem in section 6.

\begin{proof}[Proof of Theorem \ref{thm51}]
The proof is almost the same as that of \cite[Theorem 3.1]{V}. Let $\delta>0$ and $g_\delta(x)=\max\{\delta^{-2}(\delta-|x|),0\}$. Let $$g(x)=\sum\limits_{j\in\mathbb Z}g_\delta(x+Kj).$$ On the one hand, since $g(x)$ has most mass on the points $\{Kj|j\in\mathbb Z\}$, we know that $$\left|\int_0^Tg(t)f(qu(t))dt-\sum\limits_{\substack{j\in\mathbb Z\\0\leq Kj<T}}f(qu(Kj))\right|\leq2\|f\|_\infty+\frac TK\delta\|f\|_{\infty,1}.$$ On the other hand, since $g(x)$ is periodic, we have the Fourier expansion $$g(x)=\sum\limits_{k\in\mathbb Z}a_ke^{2\pi i kx/K}.$$ A simple calculation shows that $$\sum\limits_{k\in\mathbb Z}|a_k|=|g(0)|=\frac1\delta.$$ By Lemma \ref{l51} with characters $\psi_k=e^{2\pi ikx/K}$, we have 
\begin{eqnarray*}
\left|\int_0^Tg(t)f(qu(t))dt\right|&\leq&\sum\limits_{k\in\mathbb Z}|a_k|\left|\int_0^Te^{2\pi i kt/K}f(qu(t))dt\right|\ll\frac{T\ln^3(r+2)}{\delta r^\beta}\|f\|_{\infty,4}.
\end{eqnarray*}
Combining the calculations above, we have $$\left|\frac1{T/K}\sum\limits_{\substack{j\in\mathbb Z\\0\leq Kj<T}}f(qu(Kj))\right|\ll\left(\frac KT+\delta+\frac {K\ln^3(r+2)}{\delta r^\beta}\right)\|f\|_{\infty,4}.$$ Note that $K<T$, $r\leq T$ and $\beta<1$. Let $\delta=\sqrt{\frac{K\ln^3(r+2)}{r^\beta}}$ and we complete the proof of the theorem.\\
\end{proof}

\section{Proof of the Main Theorem}
\begin{proof}[Proof of the main theorem]
By a standard approximation argument, we may assume that $f\in C^\infty(\Gamma\backslash G)$ with $\|f\|_{\infty,4}<\infty$ and $$\int_{\Gamma\backslash G}fd\mu=0.$$ We want to find $\gamma_0>0$ depending on $\kappa_1,\dots,\kappa_k$ such that for any $0<\gamma<\gamma_0$, the main theorem holds. Note that by Taylor expansion, for any $M\in\mathbb N$ and $k\in\mathbb N,$ $$(M+k)^{1+\gamma}=M^{1+\gamma}+(1+\gamma)M^\gamma k+O(M^{\gamma-1}k^2).$$ Therefore, if $M$ is sufficiently large and $\gamma<1/2$, then the sequence $$\left\{(M+k)^{1+\gamma}\Big|\;0\leq k\leq\frac1{1+\gamma}M^{\frac12-\gamma}\right\}$$ is approximately equal to the arithmetic progression $$\left\{M^{1+\gamma}+(1+\gamma)M^\gamma k\Big|\;0\leq k\leq\frac1{1+\gamma}M^{\frac12-\gamma}\right\}$$ since $$O(M^{\gamma-1}k^2)\leq O(M^{\gamma-1}(M^{\frac12-\gamma})^2)=O(M^{-\gamma})\to0$$ as $M\to\infty$.\\

By Proposition \ref{p41}, we know that for any $\epsilon>0$ and any $N>0$, 
\begin{eqnarray*}
\frac1N\left|\left\{n\in[1,N]\Big|p\left(\begin{array}{cc}n^{\frac14} & n^{\frac34+\gamma} \\0 & n^{-\frac14}\end{array}\right)\in S_{\theta(n,\epsilon)}\right\}\right|\leq C_0\epsilon,
\end{eqnarray*}
where $\theta(n)=\epsilon \min\{n^{-\frac14+\frac1{\kappa_j+4}}|j=1,2,\dots,k\}$. Set 
\begin{eqnarray*}
 B=\left\{n\in\mathbb N\Big|p\left(\begin{array}{cc}n^{\frac14} & n^{\frac34+\gamma} \\0 & n^{-\frac14}\end{array}\right)\in S_{\theta(n,\epsilon)}^c\right\}.
\end{eqnarray*}
  We proceed as follows. Fix $\epsilon>0$. We pick the first element $M_1\in\mathbb N$ which lies in $B$. Then we take $$P_1=\left\{M_1+k\Bigg|\;0\leq k\leq\frac1{1+\gamma}M_1^{\frac12-\gamma}\right\}.$$ Next we pick the first element $M_2\in\mathbb N$ which appears after $P_1$ and lies in $B$, and we take $$P_2=\left\{M_2+k\Big|\;0\leq k\leq\frac1{1+\gamma}M_2^{\frac12-\gamma}\right\}.$$ Then we pick the first element $M_3\in\mathbb N$ which appears after $P_2$ and lies in $B$, and so on. In this manner, we get pieces $P_1,P_2,\dots$ in $\mathbb N$ and by our choices of $M_1,M_2,\dots$, we know that $$B\subset P_1\cup P_2\cup\dots$$ and hence for any $N>0$
\begin{eqnarray}\label{eq5}
\frac1N\left|[1,N]\setminus(P_1\cup P_2\cup\dots)\right|\leq C_0\epsilon.
\end{eqnarray}

Now we consider each of the pieces $P_i$. From the discussion above, we know that $\{n^{1+\gamma}|n\in P_i\}$ is approximated by the arithmetic progression $$\tilde P_i=\left\{M_i^{1+\gamma}+(1+\gamma)M_i^\gamma k\Big|\;0\leq k\leq\frac1{1+\gamma}M_i^{\frac12-\gamma}\right\}.$$ 

We would like to apply Theorem \ref{thm51} with $T=M_i^{1/2}$, $K=(1+\gamma)M_i^\gamma$ and $q=q_i:=pu(M_i^{1+\gamma})$ for sufficiently large $i$. So first we have to check that $r_i:=r(q_i,M_i^{\frac12})\geq1$ for sufficiently large $i$. We compute that
\begin{eqnarray*}
g_{(\log M_i)/2}(q_i)=p\left(\begin{array}{cc}1 & M_i^{1+\gamma} \\0 & 1\end{array}\right)\left(\begin{array}{cc}M_i^{\frac14} & 0 \\0 & M_i^{-\frac14}\end{array}\right)=p\left(\begin{array}{cc}M_i^\frac14 & M_i^{\frac34+\gamma} \\0 & M_i^{-\frac14}\end{array}\right)\in S_{\theta(M_i,\epsilon)}^c,
\end{eqnarray*}
by our choice of $M_i\in B$. By definition \ref{def:S}  and equation \eqref{eq:theta} , we have $$d(g_{(\log M_i)/2}(q_i))\geq\theta(M_i,\epsilon)=\epsilon \min\{M_i^{-\frac14+\frac1{\kappa_j+4}}|j=1,2,\dots,k\}.$$ By Lemma \ref{l32}, $e^{-\dist(q)}\sim d(q)^2$. 
Hence 
\begin{eqnarray}\label{eq8}
r_i=M_i^{1/2}e^{-\dist(g_{(\log M_i)/2}(q_i))}\gg\epsilon^2\min\{M_i^{2/(\kappa_j+4)}|j=1,2,\dots,k\}.
\end{eqnarray}
This implies that $r_i\to\infty$ as $i\to\infty$, since $M_i\to\infty$ by our choices of $M_i$'s.

By Theorem \ref{thm51} with $T=M_i^{1/2}$, $K=(1+\gamma)M_i^\gamma$, $q_i=pu(M_i^{1+\gamma})$ and $r_i=r(M_i^\frac12,q_i)$, we have
\begin{eqnarray}\label{eq6}
\left|\frac1{|\tilde P_i|}\sum\limits_{n\in\tilde P_i}f(pu(n))\right|&=&\left|\frac1{\lfloor M_i^{1/2}/(1+\gamma)M_i^\gamma\rfloor}\sum\limits_{0\leq (1+\gamma)M_i^\gamma k<M_i^{\frac12}}f(q_iu((1+\gamma)M_i^\gamma k))\right|\nonumber\\
&\ll&\frac {((1+\gamma)M_i^\gamma)^\frac12\ln^\frac32(r_i+2)}{r_i^\frac\beta2}\|f\|_{\infty,4}.
\end{eqnarray}
Since $M_i\to\infty$, according to inequalities (\ref{eq8}) and (\ref{eq6}), as long as $\gamma<\min\{2\beta/(\kappa_j+4)|j=1,2,\dots,k\}$, we have $$\left|\frac1{|\tilde P_i|}\sum\limits_{n\in\tilde P_i}f(pu(n))\right|\to0$$ and hence by the fact that $\{n^{1+\gamma}|n\in P_i\}$ is approximated by $\tilde P_i$, i.e., for $0\leq k\leq\frac1{1+\gamma}M_i^{\frac12-\gamma}$, 
$$|f((M_i+k)^{1+\gamma})-f(M_i^{1+\gamma}+(1+\gamma)M_i^\gamma k)|\ll M_i^{-\gamma}\|f\|_{\infty,1}$$
and $\|f\|_{\infty,1}<\infty$ we obtain 
\begin{eqnarray}\label{eq9}
\left|\frac1{|P_i|}\sum\limits_{n\in P_i}f(pu(n^{1+\gamma}))\right|\to0
\end{eqnarray}
as $i\to\infty$. By formula (\ref{eq5}), the proportion in $[1,N]$ which is not covered by $P_i$'s is small relative to $N$. Also observe that for the $P_i$'s which intersect $[1,N]$, their lengths are small relative to $N$. Therefore, by (\ref{eq9}) we have
\begin{eqnarray*}
&{}&\limsup\limits_{N\to\infty}\left|\frac1N\sum\limits_{n=0}^{N-1}f(pu(n^{1+\gamma}))\right|\\
&\leq&\limsup\limits_{N\to\infty}\left|\frac1N\sum\limits_{n\in[1,N]\setminus(\bigcup P_i)}f(pu(n^{1+\gamma}))\right|+\limsup\limits_{N\to\infty}\left|\frac1N\sum\limits_{n\in[1,N]\cap(\bigcup P_i)}f(pu(n^{1+\gamma}))\right|\\
&\leq&\limsup\limits_{N\to\infty}\left|\frac1N\sum\limits_{n\in[1,N]\setminus(\bigcup P_i)}f(pu(n^{1+\gamma}))\right|+\limsup\limits_{N\to\infty}\left|\frac1N\sum\limits_{[1,N]\cap P_i\neq\emptyset}\sum\limits_{n\in P_i}f(pu(n^{1+\gamma}))\right|\\
&\leq&C_0\epsilon\|f\|_{\infty,0}+0=C_0\epsilon\|f\|_{\infty,0}.
\end{eqnarray*}
Let $\epsilon\to0$ and we complete the proof of the main theorem with $\gamma_0=\min\{2\beta/(\kappa_j+4)|j=1,2,\dots,k\}$.\\
\end{proof}

\section{Further Discussions}
In the introduction, we define the Diophantine condition on a point in $\Gamma\backslash\SL(2,\mathbb R)$. 
Now let $$S_{\kappa_1,\kappa_2,\dots,\kappa_k}=\{p\in\Gamma\backslash\SL(2,\mathbb R)|p\text{ is Diophantine of type }(\kappa_1,\kappa_2,\dots,\kappa_k)\}.$$
Then we can calculate the Hausdorff dimension of the complement of $S_{\kappa_1,\kappa_2,\dots,\kappa_k}$. In fact, we have the following 
\begin{theorem}\label{thm:dimension}
We have $$\dim_HS_{\kappa_1,\kappa_2,\dots,\kappa_k}^c=2+\frac2{\min\{\kappa_j+1|1\leq j\leq k\}}.$$ If $\min\{\kappa_1,\kappa_2,\dots,\kappa_k\}=1$, then $S_{\kappa_1,\kappa_2,\dots,\kappa_k}$ has zero Lebesgue measure but has full Hausdorff dimension.
\end{theorem}
\begin{remark}
Note that the Diophantine type remains constant on any weak unstable leaf of $\{g_t\}_{t>0}$. Therefore the set of non Diophantine points on any strong stable leaf has zero Hausdorff dimension. We will give a different proof of this theorem in Appendix A.
\end{remark}
\begin{proof}
For each cusp $\eta_j$ $(1\leq j\leq k)$, we define $S_{j,\kappa}$ to be the subset of points $p\in\Gamma\backslash\SL(2,\mathbb R)$ satisfying the condition that there exist $\mu,\nu>0$ such that for every point $\left(\begin{array}{c}a \\b\end{array}\right)\in m_j(\pi_j^{-1}(p))$, either $|b|\geq\mu$ or $|a|^\kappa|b|\geq\nu$. Here $\mu$ and $\nu$ depend on $p$. Then by definition, we have $$S_{\kappa_1,\kappa_2,\dots,\kappa_k}=S_{1,\kappa_1}\cap S_{2,\kappa_2}\cap\dots\cap S_{k,\kappa_k}$$
and hence $$S_{\kappa_1,\kappa_2,\dots,\kappa_k}^c=S_{1,\kappa_1}^c\cup S_{2,\kappa_2}^c\cup\dots\cup S_{k,\kappa_k}^c.$$ Let $\kappa_0=\min\{\kappa_1,\kappa_2,\dots,\kappa_k\}$. Note that by Lemma \ref{l33},
$$S_{\kappa_0,\kappa_0,\dots,\kappa_0}\subset S_{\kappa_1,\kappa_2,\dots,\kappa_k}.$$ Therefore we get
$$\bigcap_{j=1}^k S_{j,\kappa_1}^c\cup\bigcap_{j=1}^k S_{2,\kappa_2}^c\cup\dots\cup\bigcap_{j=1}^k S_{k,\kappa_k}^c\subset S_{\kappa_1,\kappa_2,\dots,\kappa_k}^c\subset S_{\kappa_0,\kappa_0,\dots,\kappa_0}^c.$$ By Theorem 2 and Theorem 3 in \cite{MP}, for any $\kappa\geq1$ we have $$\dim_HS_{\kappa,\kappa,\dots,\kappa}^c=2+\frac2{\kappa+1}\text{   and   }\dim_H\bigcap_{j=1}^k S_{j,\kappa}^c=2+\frac2{\kappa+1}.$$ This implies that 
\begin{eqnarray*}\dim_HS_{\kappa_1,\kappa_2,\dots,\kappa_k}^c&=&\dim_HS_{\kappa_0,\kappa_0,\dots,\kappa_0}^c=\max\left\{\dim_H\bigcap_{j=1}^k S_{j,\kappa_i}^c\Big|1\leq i\leq k\right\}\\
&=&2+\frac2{\kappa_0+1}=2+\frac2{\min\{\kappa_j+1|1\leq j\leq k\}}.
\end{eqnarray*}

For the second statement, if $\min\{\kappa_1,\kappa_2,\dots,\kappa_k\}=1$, then by Lemma \ref{l33} and the ergodicity of the geodesic flow on $\Gamma\backslash\SL(2,\mathbb R)$, we know that $S_{\kappa_1,\kappa_2,\dots,\kappa_k}$ has zero Haar measure. Since $$S_{1,1,\dots,1}\subset S_{\kappa_1,\kappa_2,\dots,\kappa_k}$$ and by Theorem 1.1 in \cite{KM} $S_{1,1,\dots,1}$ has full Hausdorff dimension, this implies that $S_{\kappa_1,\kappa_2,\dots,\kappa_k}$ has full Hausdorff dimension.
\end{proof}

Finally, using the same argument as in section 4, we can actually prove that if $p$ is Diophantine 
of type $(\kappa_1,\kappa_2,\dots,\kappa_k)$ with all $\kappa_j<3$ and $0\leq\gamma<1/4$, then for any $\epsilon>0$, there exists a compact subset $K_\epsilon\subset\Gamma\backslash\SL(2,\mathbb R)$ such that for all $T\geq 0$, 
$$\frac1T\left\{x\in[1,T]\Big|p\left(\begin{array}{cc}x^{\frac14} & x^{\frac34+\gamma} \\0 & x^{-\frac14}\end{array}\right)\in K_\epsilon\right\}
\geq1-\epsilon.$$
Then using the arguments of \cite{DS} and \cite[Proposition 4.1]{Sh}, we get
\begin{theorem}
If $p$ is Diophantine of type $(\kappa_1,\kappa_2,\dots,\kappa_k)$ with all $\kappa_j<3$ and $0\leq\gamma<1/4$, then the trajectory 
$$\left\{p\left(\begin{array}{cc}x^{\frac14} & x^{\frac34+\gamma} \\0 & x^{-\frac14}\end{array}\right)\Big| x\geq1\right\}$$
is equidistributed in $\Gamma\backslash\SL(2,\mathbb R)$.
\end{theorem}

\begin{appendices}
\section{Hausdorff Dimension of Non-Diophantine Points in Finite-Volume Quotients of $\SL(2,\mathbb R)$}

\subsection{Introduction and preliminaries}
Here we will give a a different proof of Theorem \ref{thm:dimension} using results of lattice counting problem, that is, 
\begin{theorem} \label{athm}
$$\dim_HS_{\kappa_1,\kappa_2,\dots,\kappa_k}^\textrm{c}=2+\frac2{\min\{\kappa_j+1|1\leq j\leq k\}}.$$
\end{theorem}

We will denote by $A\ll B$ if there exists a constant $C>0$ such that $A\leq CB$. We will specify the constant $C$ in the contexts. If $A\ll B$ and $B\ll A$, we will write $A\sim B$. We will denote the diameter of a set $E$ by $\text{diam}(E)$.

To prove Theorem \ref{athm} we need some preliminaries. Readers may refer to \cite{KM}. Let $X$ be a Riemannian manifold, $m$ a volume form and $E$ a compact subset of $X$. A countable collection $\mathcal A$ of compact subsets of $E$ is said to be tree-like if $\mathcal A$ is the union of finite subcollections $\mathcal A_j$ such that 
\begin{enumerate}
\item $\mathcal A_0=\{E\}$.
\item For any $j$ and $A,B\in\mathcal A_j$, either $A=B$ or $A\cap B=\emptyset$.
\item For any $j$ and $B\in\mathcal A_{j+1}$, there exists $A\in\mathcal A_j$ such that $B\subset A$.
\item $d_j(\mathcal A):=\sup_{A\in\mathcal A_j}\text{diam}(A)\to 0$ as $j\to\infty$.
\end{enumerate}

We write $\mathbf A_j=\bigcup_{A\in\mathcal A_j}A$ and define $\mathbf A_\infty=\bigcap_{j\in\mathbb N}\mathbf A_j.$ Moreover, we define $$\Delta_j(\mathcal A)=\inf_{B\in\mathcal A_j}\frac{m(\mathbf A_{j+1}\cap B)}{m(B)}.$$ The following theorem gives a way to estimate the Hausdorff dimension of $\mathbf A_\infty$.

\begin{theorem}[\cite{M}, \cite{U} or \cite{KM}]\label{athm21}
Let $(X,m)$ be a Riemannian manifold. Assume that there exist constants $D>0$ and $k>0$ such that $$m(B(x,r))\leq Dr^k$$ for any $x\in X$. Then for any tree-like collection $\mathcal A$ of subsets of $E$
$$\dim_H(\mathbf A_\infty)\geq k-\limsup_{j\to\infty}\frac{\sum_{i=0}^j\log(\frac1{\Delta_i(\mathcal A)})}{\log(\frac1{d_{j+1}(\mathcal A)})}$$
\end{theorem}

\subsection{Some properties of lattices points in $\mathbb R^2$}
In this section, we will show some lemmas which will be used in the proof of Theorem \ref{athm}. For each cusp $\eta_j$ $(1\leq j\leq k)$, we define $S_{j,\kappa}$ to be the subset of points $p\in\Gamma\backslash\SL(2,\mathbb R)$ satisfying the condition that there exist $\mu,\nu>0$ such that for every point $\left(\begin{array}{c}a \\b\end{array}\right)\in m_j(\pi_j^{-1}(p))$, either $|b|\geq\mu$ or $|a|^\kappa|b|\geq\nu$. Here $\mu$ and $\nu$ depend on $p$. Then by definition, we have
$$S_{\kappa_1,\kappa_2,\dots,\kappa_k}^\textrm{c}=S_{1,\kappa_1}^\textrm{c}\cup S_{2,\kappa_2}^\textrm{c}\cup\dots\cup S_{k,\kappa_k}^\textrm{c}$$ 
and hence $$\dim_HS_{\kappa_1,\kappa_2,\dots,\kappa_k}^\textrm{c}=\max\{\dim_HS_{j,\kappa_j}^\textrm{c}|1\leq j\leq k\}.$$
Therefore, to prove Theorem~\ref{athm}, it suffices to prove $$\dim_HS_{j,\kappa_j}^\textrm{c}=2+\frac2{\kappa_j+1}.$$
In the rest of this part, we will consider  $S_{j,\kappa}^\textrm{c}$ for a fixed cusp $\eta_j$. Without loss of generality, we may assume that $\sigma_j=e$, $\eta_j=i\infty$ and that
$$\left(\begin{array}{cc}1 & 1 \\0 & 1\end{array}\right)\in\Gamma.$$ Since $\Gamma\cap N\neq\{e\}$, this implies that $\Gamma e_1$ is a discrete subset in $\mathbb R^2$. The following lemmas concern some properties of lattice points in $\Gamma e_1\subset\mathbb R^2$.

\begin{lemma}\label{al31}
There exists a constant $C>0$ such that for any $(\alpha,\beta)\in\Gamma e_1$ we have $|\beta|\geq C$ or $\beta=0$.
\end{lemma}
\begin{proof}
We know that $\Gamma e_1$ is discrete in $\mathbb R^2$. So there is a constant $C>0$ such that for any point $(\alpha,\beta)\in\Gamma e_1$ we have $$\|(\alpha,\beta)\|\geq 2C$$ where $\|\cdot\|$ is the standard Euclidean norm. Suppose that there exists $(\alpha_0,\beta_0)\in\Gamma e_1$ with $0<|\beta_0|<C$. Then there exists an integer $n\in\mathbb Z$ such that $$|\alpha_0+n\beta_0|<\beta_0.$$ Since $\left(\begin{array}{cc}1 & 1 \\0 & 1\end{array}\right)\in\Gamma$, we have $$\left(\begin{array}{cc}1 & n \\0 & 1\end{array}\right)\left(\begin{array}{c}\alpha_0 \\\beta_0\end{array}\right)=\left(\begin{array}{c}\alpha_0+n\beta_0 \\\beta_0\end{array}\right)\in\Gamma e_1$$ and $$\|(\alpha_0+n\beta_0,\beta_0)\|\leq\sqrt 2 C<2C$$ which contradicts the definition of $C$. This completes the proof of the lemma.
\end{proof}

\begin{lemma}\label{al32}
There exists a constant $C>0$ such that for any two distinct points $(\alpha_1,\beta_1)$ and $(\alpha_2,\beta_2)$ in $\Gamma e_1$ we have $$|\alpha_1\beta_2-\alpha_2\beta_1|\geq C.$$
\end{lemma}
\begin{proof}
Now let $\gamma_1,\gamma_2\in\Gamma$ be such that 
$$\gamma_1=\left(\begin{array}{cc}\alpha_1 & * \\\beta_1 & *\end{array}\right)
\quad\gamma_2=\left(\begin{array}{cc}\alpha_2 & * \\\beta_2 & *\end{array}\right).$$
Then we have $$\gamma_1^{-1}\gamma_2 e_1=\left(\begin{array}{c}* \\\alpha_1\beta_2-\alpha_2\beta_1\end{array}\right).$$
Note that $(\alpha_1,\beta_1)$ and $(\alpha_2,\beta_2)$ are distinct and hence $\alpha_1\beta_2-\alpha_2\beta_1\neq0$. By Lemma \ref{al31}, we conclude that $$|\alpha_1\beta_2-\alpha_2\beta_1|\geq C$$ for some $C>0$.
\end{proof}
\begin{remark}
We will fix this constant $C$ for later use. Note that by the definition of $C$, for any point $(\alpha,\beta)\in\Gamma e_1$ we have $\|(\alpha,\beta)\|\geq 2C$.
\end{remark}

\begin{definition}
For $l>0$ and $0\leq\theta_1\leq\theta_2<2\pi$, we define the subset of $\mathbb R^2$ $$S(l,\theta_1,\theta_2):=\{(x,y)\in\mathbb R^2|l\leq r\leq 2l,\theta_1<\theta<\theta_2\}$$ where $(r,\theta)$ are the polar coordinates of $(x,y)$.
\end{definition}

\begin{theorem}[\cite{EM}, \cite{GOS}]\label{athm31}
We have $\#|\Gamma e_1\cap S(l,\theta_1,\theta_2)|\sim l^2(\theta_2-\theta_1)$ as $l\to\infty$.
\end{theorem}

\begin{lemma}\label{al33}
Fix $C>0$ in Lemma \ref{al32} and let $\kappa\geq1$. There exists a constant $C_0>0$ with the following property: for any $(\alpha,\beta)\in\Gamma e_1$ with $0<\frac\alpha\beta<1$, there exists a large constant $L_{(\alpha,\beta)}>0$ such that for any $l>L_{(\alpha,\beta)}$ the interval $\left[\frac\alpha\beta-\frac C{18}\cdot\frac1{\beta^{\kappa+1}},\frac\alpha\beta+\frac C{18}\cdot\frac1{\beta^{\kappa+1}}\right]$ contains at least $C_0l^2/{\beta^{\kappa+1}}$ many disjoint subintervals $\left[\frac{\tilde\alpha}{\tilde\beta}-\frac C{18}\cdot\frac1{\tilde\beta^{\kappa+1}},\frac{\tilde\alpha}{\tilde\beta}+\frac C{18}\cdot\frac1{\tilde\beta^{\kappa+1}}\right]$ where $(\tilde\alpha,\tilde\beta)\in\Gamma e_1\cap S(l,\frac\pi4,\frac\pi2)$.
\end{lemma}
\begin{proof}
Suppose that $\left[\frac\alpha\beta-\frac C{18}\cdot\frac1{\beta^{\kappa+1}},\frac\alpha\beta+\frac C{18}\cdot\frac1{\beta^{\kappa+1}}\right]$ contains two subintervals
$$\left[\frac{\tilde\alpha}{\tilde\beta}-\frac C{18}\cdot\frac1{\tilde\beta^{\kappa+1}},\frac{\tilde\alpha}{\tilde\beta}+\frac C{18}\cdot\frac1{\tilde\beta^{\kappa+1}}\right]\text{ and }\left[\frac{\tilde\gamma}{\tilde\delta}-\frac C{18}\cdot\frac1{\tilde\delta^{\kappa+1}},\frac{\tilde\gamma}{\tilde\delta}+\frac C{18}\cdot\frac1{\tilde\delta^{\kappa+1}}\right]$$ 
where $(\tilde\alpha,\tilde\beta)$ and $(\tilde\gamma,\tilde\delta)$ are two distinct points in $\Gamma e_1\cap S(l,\frac\pi4,\frac\pi2)$. By Lemma \ref{l32}, we have 
\begin{eqnarray*}
&{}&\left|\frac{\tilde\alpha}{\tilde\beta}-\frac{\tilde\gamma}{\tilde\delta}\right|=\frac{|\tilde\alpha\tilde\delta-\tilde\beta\tilde\gamma|}{|\tilde\beta\tilde\delta|}\geq\frac C{|\tilde\beta\tilde\delta|}\geq\frac C{4l^2}\\
&=&\frac C{16}\left(\frac1{(l/\sqrt 2)^2}+\frac1{(l/\sqrt 2)^2}\right)\geq\frac C{16}\left(\frac1{\tilde\beta^2}+\frac1{\tilde\delta^2}\right)\\
&\geq&\frac C{16}\left(\frac1{\tilde\beta^{\kappa+1}}+\frac1{\tilde\delta^{\kappa+1}}\right).
\end{eqnarray*}
This implies that any two such subintervals are disjoint, and hence to prove the lemma it suffices to prove that in the interval $\left[\frac\alpha\beta-\frac C{18}\cdot\frac1{\beta^{\kappa+1}},\frac\alpha\beta+\frac C{18}\cdot\frac1{\beta^{\kappa+1}}\right]$ there are at least $C_0l^2/{\beta^{\kappa+1}}$ many points of the form $\tilde\alpha/\tilde\beta$ where $(\tilde\alpha,\tilde\beta)\in\Gamma e_1\cap S(l,\frac\pi4,\frac\pi2)$. We have 
\begin{eqnarray*}
&{}&\frac{\tilde\alpha}{\tilde\beta}\in\left[\frac\alpha\beta-\frac C{18}\cdot\frac1{\beta^{\kappa+1}},\frac\alpha\beta+\frac C{18}\cdot\frac1{\beta^{\kappa+1}}\right]\\
&\iff&\arg(\tilde\alpha,\tilde\beta)\in\left[\arccot\left(\frac\alpha\beta+\frac C{18}\cdot\frac1{\beta^{\kappa+1}}\right),\arccot\left(\frac\alpha\beta-\frac C{18}\cdot\frac1{\beta^{\kappa+1}}\right)\right].
\end{eqnarray*}
Since $\left|\arccot\left(\frac\alpha\beta+\frac C{18}\cdot\frac1{\beta^{\kappa+1}}\right)-\arccot\left(\frac\alpha\beta-\frac C{18}\cdot\frac1{\beta^{\kappa+1}}\right)\right|\sim\frac1{\beta^{\kappa+1}}$, by Theorem \ref{athm31} we know that the number of points in $S\left(l,\arccot\left(\frac\alpha\beta+\frac C{18}\cdot\frac1{\beta^{\kappa+1}}\right),\arccot\left(\frac\alpha\beta-\frac C{18}\cdot\frac1{\beta^{\kappa+1}}\right)\right)$ is asymptotically equal to $l^2/\beta^{\kappa+1}$ up to a constant. Note that the implicit constant is absolute since $0<\alpha/\beta<1$. This completes the proof of the lemma.
\end{proof}

\subsection{Hausdorff dimension of the subset of non-Diophantine points}
In this section, we will give a proof of Theorem \ref{athm}. We need some preparations.
\begin{definition}
We say that $x\in\mathbb R$ is Diophantine of type $\kappa$ with respect to $\Gamma e_1$ if there exists a constant $\tilde C>0$ such that for any $(\alpha,\beta)\in\Gamma e_1$ with $\beta\neq0$ we have $$|\beta|^\kappa|x\beta-\alpha|\geq \tilde C.$$ We denote by $S_\kappa$ the subset of $\mathbb R$ of all Diophantine numbers of type $\kappa$ with respect to $\Gamma e_1$.
\end{definition}

\begin{lemma}\label{al41}
Let $p=\Gamma\left(\begin{array}{cc}a & b \\c & d\end{array}\right)\in\Gamma\backslash\SL(2,\mathbb R)$ with $c\neq0$. Then $p\in S_{j,\kappa}^\textrm{c}$ if and only if $a/c\in S_\kappa^\textrm{c}$. 
\end{lemma}
\begin{proof}
We have
$$m_j(\pi_j^{-1}(p))=\left(\begin{array}{cc}d & -b \\-c & a\end{array}\right)\Gamma e_1
=\left\{\left(\begin{array}{c}d\alpha-b\beta \\-c\alpha+a\beta\end{array}\right)\bigg|\left(\begin{array}{c}\alpha \\ \beta\end{array}\right)\in\Gamma e_1\right\}.$$
By the definition of $S_{j,\kappa}$, if $p\in S_{j,\kappa}^\textrm{c}$, then there exist infinitely many $(\alpha,\beta)\in\Gamma e_1$ such that $$|a\beta-c\alpha|\to0\text{ and }|d\alpha-b\beta|^{\kappa}|a\beta-c\alpha|\to 0.$$ By the discreteness of $\Gamma e_1$, this implies that $|\beta|\to\infty$. Note that $$|d\alpha-b\beta|=\frac{|cd\alpha-cb\beta|}c=\frac{|cd\alpha-(ad-1)\beta|}c=\frac{|d(c\alpha-a\beta)+\beta|}c.$$ Therefore we have $|d\alpha-b\beta|\sim|\beta|$ and $a/c\in S_\kappa^\textrm{c}$. Here the implicit constant in $\sim$ depends on $p$. 

Conversely, if $a/c\in S_\kappa^\textrm{c}$, then there exist infinitely many $(\alpha,\beta)\in\Gamma e_1$ with $\beta\neq0$ such that $$\left|\beta\right|^\kappa\left|\frac ac\beta-\alpha\right|\to0.$$
By Lemma \ref{al31}, this implies that $$|a\beta-c\alpha|\to0$$ and consequently $$|\beta|\to\infty\text{ and }|d\alpha-b\beta|=\frac{|d(c\alpha-a\beta)+\beta|}c\sim|\beta|.$$ Hence we have $$|a\beta-c\alpha|\to0,\quad|d\alpha-b\beta|^\kappa|a\beta-c\alpha|\to0$$ and $p\in S_{j,\kappa}^\textrm{c}.$ This completes the proof of the lemma.
\end{proof}

\begin{proof}[Proof of Theorem \ref{athm}]
From the discussions above, we know that in order to prove Theorem \ref{athm} it is enough to show that $$\dim_HS_{j,\kappa}^\textrm{c}=2+\frac2{\kappa+1}.$$
By Lemma \ref{l41} and the fact that the subset $\left\{\Gamma\left(\begin{array}{cc}a & b \\0 & a^{-1}\end{array}\right)\right\}\subset\Gamma\backslash\SL(2,\mathbb R)$ has dimension 2, it suffices to prove that $$\dim_H S_\kappa^\textrm{c}=\frac2{\kappa+1}.$$
In the rest of this section we will prove this formula.

Since $\left(\begin{array}{cc}1 & 1 \\0 & 1\end{array}\right)\in\Gamma$, for any $n\in\mathbb Z$ we have $$S_\kappa^\textrm{c}\cap(n,n+1)=n+S_\kappa^\textrm{c}\cap(0,1).$$ Therefore, we only need to compute the Hausdorff dimension of $S_\kappa^\textrm{c}\cap(0,1)$. For the upper bound, by the definition of $S_\kappa$, we can construct an open cover $$\left\{I_{(\alpha,\beta)}=\left(\frac\alpha\beta-\frac1{\beta^{\kappa+1}},\frac\alpha\beta+\frac1{\beta^{\kappa+1}}\right)\Bigg|(\alpha,\beta)\in\Gamma e_1,\alpha/\beta\in(0,1)\right\}\supseteq S_\kappa^\textrm{c}\cap(0,1).$$ 
For $\delta>0$ by Theorem \ref{athm31} we have 
\begin{eqnarray*}
&{}&\sum_{\substack{(\alpha,\beta)\in\Gamma e_1\\\alpha/\beta\in(0,1)}}\text{diam}(I_{(\alpha,\beta)})^\delta\\
&\ll&\sum_{n=1}^\infty\sum_{(\alpha,\beta)\in\Gamma e_1\cap S(2^nC,\frac\pi4,\frac\pi2)}\frac1{\beta^{\delta(\kappa+1)}}\\
&\ll&\sum_{n=1}^\infty \frac{2^{2n}}{2^{n\delta(\kappa+1)}}=\sum_{n=1}^\infty \frac1{2^{n(\delta(\kappa+1)-2)}}.
\end{eqnarray*}
If $\delta>2/(\kappa+1)$, then $\sum_{\substack{(\alpha,\beta)\in\Gamma e_1\\\alpha/\beta\in(0,1)}}\text{diam}(I_{(\alpha,\beta)})^\delta$ converges and hence by properties of Hausdorff dimension we have $$\dim _HS_\kappa^\textrm{c}\cap(0,1)\leq\frac2{\kappa+1}.$$

For the lower bound, let $\epsilon>0$ be fixed and we construct a tree-like set in $S_\kappa^\textrm{c}\cap(0,1)$ as the intersection of closed subsets in $[0,1]$ by induction. Let $\mathcal A_0=\{[0,1]\}$ and $\mathbf A_0=[0,1]$. Let $l_1$ be a sufficiently large number and define $$\mathcal A_1=\left\{\left[\frac\alpha\beta-\frac C{18}\cdot\frac1{\beta^{\kappa+\epsilon+1}},\frac\alpha\beta+\frac C{18}\cdot\frac1{\beta^{\kappa+\epsilon+1}}\right]\Bigg|(\alpha,\beta)\in\Gamma e_1\cap S(l_1,\frac\pi4,\frac\pi2)\right\}$$ and $\mathbf A_1=\bigcup\mathcal A_1$. Suppose that we find $l_1<l_2<\dots<l_j$ and construct families $\mathcal A_j,\mathcal A_{j-1},\dots,\mathcal A_0$ and closed subsets $\mathbf A_j\subseteq\mathbf A_{j-1}\subseteq\dots\subseteq \mathbf A_1\subseteq \mathbf A_0$. Now by Lemma \ref{al33}, we can find a sufficiently large $l_{j+1}>0$ such that 
\begin{enumerate}
\item $\log l_{j+1}\geq j^2\log(l_jl_{j-1}\dots l_1)$.
\item For every $\left[\frac\alpha\beta-\frac C{18}\cdot\frac1{\beta^{\kappa+\epsilon+1}},\frac\alpha\beta+\frac C{18}\cdot\frac1{\beta^{\kappa+\epsilon+1}}\right]\in \mathcal A_j$, it contains at least $C_0l_{j+1}^2/{l_j^{\kappa+\epsilon+1}}$ subintervals (since $\beta\sim l_j$) of the form $\left[\frac{\tilde\alpha}{\tilde\beta}-\frac C{18}\cdot\frac1{{\tilde\beta}^{\kappa+\epsilon+1}},\frac{\tilde\alpha}{\tilde\beta}+\frac C{18}\cdot\frac1{{\tilde\beta}^{\kappa+\epsilon+1}}\right]$ with $(\tilde\alpha,\tilde\beta)\in\Gamma e_1\cap S(l_{j+1},\frac\pi4,\frac\pi2)$.
\end{enumerate}
We denote the family of all these new subintervals by $\mathcal A_{j+1}$ as $\left[\frac\alpha\beta-\frac C{18}\cdot\frac1{\beta^{\kappa+\epsilon+1}},\frac\alpha\beta+\frac C{18}\cdot\frac1{\beta^{\kappa+\epsilon+1}}\right]$ runs through all the intervals in $\mathcal A_j$ and let $\mathbf A_{j+1}=\bigcup\mathcal A_{j+1}$.

Now we take $\mathbf A_\infty=\bigcap_{j=0}^\infty\mathbf A_j$ and $\mathcal A=\bigcup_{j=0}^\infty\mathcal A_j$. From the construction of $\mathbf A_j$'s and the definition of $S_\kappa$, we know that $\mathbf A_\infty\subseteq S_\kappa^\textrm{c}\cap(0,1)$. Also we have $$\Delta_j(\mathcal A)\sim\frac{l_{j+1}^2}{l_j^{\kappa+\epsilon+1}}\cdot\frac1{l_{j+1}^{\kappa+\epsilon+1}}\text{ and }d_j(\mathcal A)\sim\frac1{l_j^{\kappa+\epsilon+1}}.$$ Therefore by Theorem \ref{athm21}, we have
\begin{eqnarray*}
&{}&\dim_HS_\kappa^\textrm{c}\cap(0,1)\\
&\geq&\dim_H\mathbf A_\infty\\
&\geq&1-\limsup_{j\to\infty}\frac{-\sum_{i=1}^j\log(l_{i+1}^2/{(l_il_{i+1})^{\kappa+\epsilon+1}})}{\log l_{j+1}^{\kappa+\epsilon+1}}\\
&=&1-\limsup_{j\to\infty}\frac{(\kappa+\epsilon+1)\log l_1+\sum_{i=2}^j2(\kappa+\epsilon)\log l_i+(\kappa+\epsilon-1)\log l_{j+1}}{(\kappa+\epsilon+1)\log l_{j+1}}\\
&=&1-\frac{\kappa+\epsilon-1}{\kappa+\epsilon+1}=\frac2{\kappa+\epsilon+1}.
\end{eqnarray*}
Since this is true for any $\epsilon>0$, we obtain that $$\dim_HS_\kappa^\textrm{c}\cap(0,1)\geq\frac2{\kappa+1}.$$ This completes the proof of Theorem~\ref{athm}.
\end{proof}

\section{Effective Equidistribution of Abelian Horospherical Orbits in Finite-Volume Homogeneous Spaces}

\subsection{Introduction}
In this part, we will consider the effective equidistribution of horospherical orbits in homogeneous spaces. This topic has been studied well, and the present work is motivated by \cite{S} and \cite{V}. To be precise, let $\{a_t\}=\{\exp(tX)\}_{t\in\mathbb R}$ be a one parameter subgroup consisting of semisimple elements in a semisimple Lie group $G$, $\Gamma$ a lattice in $G$ and $\mu$ the Haar measure on $\Gamma\backslash G$. Let $\text{Ad}(g)$ be the adjoint action of $G$ on $\text{Lie}(G)$ induced by the action of conjugation $x\mapsto gxg^{-1}$. Let $U$ be the horospherical subgroup of $\{a_t\}$, i.e. $$U=\{g\in G| a_{-t}ga_t\to e\}.$$ The decomposition of $\text{Lie}(U)$ with respect to $\{a_t\}$ under the adjoint action is $$\text{Lie}(U)=\mathfrak g_{\alpha_1}\oplus\mathfrak g_{\alpha_2}\oplus\dots\oplus\mathfrak g_{\alpha_n}$$ where $\alpha_i$ are the roots of $\{a_t\}$, that is, $$\text{Ad}(a_t) x_i=\alpha_i(a_t)x_i$$ for any $x_i\in\mathfrak g_{\alpha_i}$. Without loss of generality, we can assume that each $\mathfrak g_{\alpha_i}$ is one-dimensional and some of these $\alpha_i$'s may be identical. We denote the exponential map from $\text{Lie}(G)$ to $G$ by $\exp$. For each $i$, fix $v_i\in\mathfrak g_{\alpha_i}$ with norm $1$ and let $B(T_1,T_2,\dots, T_n)$ be the parametrized box in $U$, i.e. $$B(T_1,T_2,\dots, T_n)=\{\exp(t_1v_1+t_2v_2+\dots+t_nv_n)|0\leq t_i\leq T_i(1\leq i\leq n)\}.$$ For any $t>0$ we define $$e^{\alpha t}:=\alpha_1(a_t)\alpha_2(a_t)\cdots \alpha_n(a_t)$$ for some $\alpha>0$ and then
 $$t^\alpha=\alpha_1(a_{\ln t})\alpha_2(a_{\ln t})\cdots \alpha_n(a_{\ln t}).$$ Also define
\begin{eqnarray*}
B(t):&=&B(\alpha_1(a_{\ln t}),\alpha_2(a_{\ln t}),\cdots,\alpha_n(a_{\ln t}))\\
&=&a_{\ln t}B(1,1,\dots,1)a_{-\ln t}.
\end{eqnarray*}
We will denote by $B_r$ the open ball of radius $r>0$ around $e$ in $G$.

\begin{definition}
For any $x\in\Gamma\backslash G$, we define the injectivity radius at $x$ by the largest number $\eta>0$ with the property that the map $$B_\eta\to xB_\eta\subset\Gamma\backslash G$$ by sending $g\in B_\eta$ to $xg\in\Gamma\backslash G$ is injective. We will denote the injectivity radius at $x$ by $\eta(x)$.
\end{definition}

We will denote by $A\ll B$ if there exists a constant $C>0$ such that $A\leq CB$. We will specify the constant $C>0$ in the contexts. If $A\ll B$ and $B\ll A$, we write $A\sim B$. We will follow the proof of Lemma 9.5 in \cite{V} and prove the following theorem

\begin{theorem}\label{bequidistribution}
Suppose that $U$ is abelian. There exist constants $a,b>0$ such that for any $f\in C^\infty(\Gamma\backslash G)$, we have $$\left|\frac1{T^\alpha}\int_{B(T)}f(xu)du-\int_{\Gamma\backslash G} fd\mu\right|\ll\frac1{T^a\eta^b}\|f\|_{\infty,l}.$$ Here $\eta=\eta(a_{\ln T}x)$ is the injectivity radius at $a_{\ln T}x$ and $\|\cdot\|_{\infty,l}$ is the $L^\infty$-Sobolev norm involving Lie derivatives of orders up to $l$ for some $l>0$. The implicit constant depends only on $\Gamma\backslash G$.
\end{theorem}
\begin{remark}
We assume that $U$ is abelian so that the proof would be simple.  We expect that the theorem would still hold if $U$ is not abelian but the calculations would become complicated. 
\end{remark}
\begin{remark}
We will always assume that $\|f\|_{\infty,l}$ is defined and finite, and $l$ is large enough so that all the theorems and arguments in this note would hold. Readers may refer to \cite{KM1} for more details about the Sobolev norm. 
\end{remark}
\begin{remark}
Theorem \ref{bequidistribution} is weaker than the theorem proved by Str\"ombergsson \cite{S} in the case of $\Gamma\backslash\SL(2,\mathbb R)$ since the exponents $a$ and $b$ here can not be optimized. But the proof would involve only mixing property of a semisimple flow and give a result for a general homogeneous space. Readers may compare $T\eta^{\frac ba}$ and the $r$-factor in the main theorem of \cite{S}.
\end{remark}

Using the same arguments as in the proof of Theorem \ref{bequidistribution}, we can prove the following
\begin{theorem}\label{bequidistribution1}
Suppose that $U$ is abelian. Let $h(u)$ be a compactly supported smooth function on $U$. Then there exist constants $a,b>0$ such that for any $f\in C^\infty(\Gamma\backslash G)$ we have $$\left|\frac1{T^\alpha}\int_Uf(xu)h(a_{-\ln T}ua_{\ln T})du-\int_{\Gamma\backslash G} fd\mu\int_Uh(u)du\right|\ll\frac1{T^a\eta^b}\|f\|_{\infty,l}.$$ Here $\eta$ and $\|\cdot\|_{\infty,l}$ are the same as in Theorem \ref{bequidistribution}. The implicit constant depends only on $h(u)$ and $\Gamma\backslash G$.
\end{theorem}

\begin{definition}
A point $p\in\Gamma\backslash G$ is called Diophantine of type $\mu$ with respect to $\{a_t\}$ if there exists a constant $C>0$ such that $$\eta(pa_t)\geq Ce^{-\mu t}$$ for all $t>0$. Also we say that an orbit $\{pa_t\}_{t\geq0}$ in $\Gamma\backslash G$ is non-divergent of order $\mu$ if there exists a constant $C>0$ such that $$\eta(pa_{t_k})\geq Ce^{-\mu t_k}$$ for infinitely many $t_k\to\infty$.
\end{definition}
\begin{remark}
Note that $\{pa_t\}_{t\geq0}$ is non-divergent of order $0$ with respect to $\{a_t\}$ if and only if $\{pa_t\}$ is non-divergent.
\end{remark}
The following is an immediate corollary of Theorem \ref{bequidistribution} and \cite{Sh}.
\begin{corollary}
Assume the conditions in Theorem \ref{bequidistribution}. If $x$ is Diophantine of type $\mu<a/b$ with respect to $\{a_t\}$, or $\{xa_t\}_{t\geq0}$ is non-divergent of order $\mu<a/b$, then $$\frac1{T^\alpha}\int_{B(T)}f(xu)du\to\int fd\mu.$$ 
Here the constants $a,b$ and the function $f$ are as in Theorem \ref{bequidistribution}.
\end{corollary}

\noindent\textbf{Acknowledge.} I would like to thank Professor Andreas Str\"ombergsson and Samuel Edwards  for many discussions. I was told that they had results about the effective equidistribution of horocycle orbits in homogeneous spaces using number theoretic tools. Here what we prove in Theorem \ref{bequidistribution} is much weaker than \cite{S} and our purpose is just to show how to use mixing property only to get a version of such result.

\subsection{Preliminaries}
In the proof of Theorem \ref{bequidistribution}, we will need the following exponential mixing property. 

\begin{theorem}[Kleinbock and Margulis \cite{KM1}]\label{bthm21}
There exists $\kappa>0$ such that for any $f,g\in C^\infty (\Gamma\backslash G)$, we have $$\left|(a_t\cdot f,g)-\int_{\Gamma\backslash G}f\int_{\Gamma\backslash G}g\right|\ll e^{-\kappa t}\|f\|_{\infty,l}\|g\|_{\infty,l}.$$ Here $(a_t\cdot f)(x)=f(xa_{-t})$ is the right translation of $f$ by $a_t$ and $\|\cdot\|_{\infty,l}$ is the same Sobolev norm as in Theorem \ref{bequidistribution}.
\end{theorem}

\subsection{Somme lemmas}
In this section, we will use the same arguments in the proof of Lemma 9.5 in \cite{V} and prove some lemmas which will be used in the proof of Theorem \ref{bequidistribution} and Theorem \ref{bequidistribution1}. 

\begin{lemma}\label{bl31}
Let $x$ be any point in $\Gamma\backslash G$. Then for every point $y\in xB(1)$ we have $$\eta(y)\sim\eta(x).$$ Here the implicit constant depends only on $G$. Generally, if $y\in xB$ for some bounded subset $B\subset U$, then the same result holds with the implicit constant depending only on $B$ and $G$.
\end{lemma}
\begin{proof}
Let $y=xu$ for some $u\in B(1)$. By definition, we know that the map $$B_{\eta(x)}\to xB_{\eta(x)}$$ is injective. This implies that $$u^{-1}B_{\eta(x)}u\to yu^{-1}B_{\eta(x)}u$$ is injective and hence $$\eta(x)\ll\eta(y)$$ for some implicit constant depending only on $B(1)$. The proof of $\eta(y)\ll\eta(x)$ is similiar. Also the proof of the general case is the same. This completes the proof of the lemma.
\end{proof}

Now we fix a positive compactly supported smooth function $g(x)$ with integral one on $\mathbb R$, and for any $n\in\mathbb N$, $\gamma>0$ and $\delta>0$ define
\begin{eqnarray*}
&{}&g_{\delta,n,\gamma}(u)\\
&=&\frac1{\delta^n}\int_0^\gamma\int_0^\gamma\cdots\int_0^\gamma g\left(\frac{u_1-t_1}\delta\right) g\left(\frac{u_2-t_2}\delta\right)\dots g\left(\frac{u_n-t_n}\delta\right)dt_1dt_2\dots dt_n.
\end{eqnarray*}
The following lemma is an immediate consequence from calculations.

\begin{lemma}\label{bl32}
We have 
\begin{enumerate}
\item $\int_{\mathbb R^n}g_{\delta,n,\gamma}(u)du=\gamma^n$.
\item $g_{\delta,n,\gamma}(u)$ is supported around the box $[0,\gamma]\times[0,\gamma]\times\dots\times[0,\gamma]$.
\item $\int_{\mathbb R^n} |g_{\delta,n,\gamma}(u)-\chi_{[0,\gamma]^n}(u)|du\ll\delta(\gamma+\delta)^{n-1}$.
\end{enumerate}
\end{lemma}

\begin{lemma}\label{bl33}
Let $y\in\Gamma\backslash G$ and $f\in C^\infty(\Gamma\backslash G)$. Assume that $\int_{\Gamma\backslash G} f d\mu=0$. Then there exist constants $a,b>0$ such that for any $t>0$ and $\gamma<\frac{\eta(y)}2$ we have
\begin{eqnarray*}
\left|\int_{B(\gamma,\gamma,\dots,\gamma)}f(yua_{-t})du\right|\ll\frac1{e^{at}\gamma^b}\|f\|_{\infty,l}.
\end{eqnarray*}
The implicit constant depends only on $\Gamma\backslash G$.
\end{lemma}
\begin{proof}
Now let $U^+$ be the unstable horospherical subgroup of $\{a_t\}$ and $Z=Z(a_t)$ be the central subgroup of $\{a_t\}$ in $G$. Then we know that $$\text{Lie}(U)\oplus\text{Lie}(U^+)\oplus\text{Lie}(Z)=\text{Lie}(G).$$ Let $\dim U=\dim U^+=n$ and $\dim Z(a_t)=m$. By Lemma \ref{bl32} and the same arguments as in Lemma 9.5 of \cite{V} we have
\begin{eqnarray*}
&{}&\int_{B(\gamma,\gamma,\dots,\gamma)}f(yua_{-t})du\\
&=&\int_{\text{Lie}(U)}f(y\exp(u)a_{-t})g_{\delta,n,\gamma}(u)du+O(\|f\|_{\infty,l})\delta(\gamma+\delta)^{n-1}\\
&=&\frac1{\delta^m\gamma^n}\iiint_{\text{Lie}(U)\times\text{Lie}(Z)\times\text{Lie}(U^+)}f(y\exp(u)a_{-t})g_{\delta,n,\gamma}(u)g_{\delta,m,\delta}(z)g_{\delta,n,\gamma}(v)dudzdv\\
&{}&+O(\|f\|_{\infty,l})\delta(\gamma+\delta)^{n-1}\\
&=&\frac1{\delta^m\gamma^n}\iiint_{\text{Lie}(G)}f(y\exp(u)\exp(z)\exp(v)a_{-t})g_{\delta,n,\gamma}(u)g_{\delta,m,\delta}(z)g_{\delta,n,\gamma}(v)dudzdv\\
&{}&+O(\|f\|_{\infty,l})(\delta(\gamma+\delta)^{n-1}+\gamma^n\max\{\delta,\gamma/e^{q t}\})\\
&=&\frac1{\delta^m\gamma^n}\int_{\Gamma\backslash G} f(xa_{-t})g_{\delta,y}(x)d\mu(x)+O(\|f\|_{\infty,l})(\delta(\gamma+\delta)^{n-1}+\gamma^n\max\{\delta,\gamma/e^{q t}\})\\
&=&\frac1{\delta^m\gamma^n}(a_t\cdot f,g_{\delta,y})+O(\|f\|_{\infty,l})(\delta(\gamma+\delta)^{n-1}+\gamma^n\max\{\delta,\gamma/e^{q t}\})
\end{eqnarray*}
where $q>0$ is a positive constant and $g_{\delta,y}$ is a function supported on the ball of radius $\eta(y)$ at $y$ in $\Gamma\backslash G$. Note that all injectivity radii have a common upper bound depending only on $\Gamma\backslash G$. By the definition of Lie derivatives, we can compute $\|g_{\delta,y}\|_{\infty,l}$ and there exists a constant $p>0$ such that $$\|g_{\delta,y}\|_{\infty,l}\ll1/\delta^p.$$ Therefore, by exponential mixing of semisimple flow (Theorem \ref{bthm21}), we have
\begin{eqnarray*}
&{}&\left|\int_{B(\gamma,\gamma,\dots,\gamma)}f(yua_{-t})du\right|\\
&\ll&\frac1{\delta^m\gamma^n}\frac1{e^{\kappa t}\delta^p}\|f\|_{\infty,l}+\|f\|_{\infty,l}(\delta(\gamma+\delta)^{n-1}+\max\{\delta,\gamma/e^{q t}\}).
\end{eqnarray*}
Let $\delta=\gamma e^{-\epsilon t}<\gamma$ for some small $\epsilon>0$ and this completes the proof of the lemma.
\end{proof}

\begin{lemma}\label{bl34}
Assume the conditions in Lemma \ref{bl33}. Let $h(u)$ be a smooth compactly supported function on $U$. Then there exist constants $a,b>0$ such that for any $t>0$ and $\gamma<\frac{\eta(y)}2$ we have
\begin{eqnarray*}
\left|\int_{B(\gamma,\gamma,\dots,\gamma)}f(yua_{-t})h(u)du\right|\ll\frac1{e^{at}\gamma^b}\|f\|_{\infty,l}\|h\|_{\infty,l}.
\end{eqnarray*}
Here $\|h(u)\|_{\infty,l}$ is the $L^\infty$-Sobolev norm involving partial derivatives of orders up to $l$ on $U$. The implicit constant depends only on $\Gamma\backslash G$.
\end{lemma}
\begin{proof}
By Lemma \ref{bl32}, we have
\begin{eqnarray*}
&{}&\int_{B(\gamma,\gamma,\dots,\gamma)}f(yua_{-t})h(u)du\\
&=&\int_{\text{Lie}(U)}f(y\exp(u)a_{-t})h(\exp(u))g_{\delta,n,\gamma}(u)du+O(\|f\|_{\infty,l}\|h\|_{\infty,l})\delta(\gamma+\delta)^{n-1}.
\end{eqnarray*}
Now the lemma follows from the same arguments as in Lemma \ref{bl33}. (In this case, we have $\|g_{\delta,y}\|_{\infty,l}\ll1/\delta^p\|h\|_{\infty,l}$ for some $p>0$.)
\end{proof}

\subsection{Effective equidistribution of abelian horospherical orbits}
In this section, we will prove Theorem \ref{bequidistribution} and Theorem \ref{bequidistribution1}.
\begin{proof}[Proof of Theorem \ref{bequidistribution}]
Without loss of generality, assume that $\int fd\mu=0$. We know that 
$$\frac1{T^\alpha}\int_{B(T)}f(xu)du=\int_{B(1)}f(xa_{\ln T}ua_{-\ln T})du.$$
By Lemma \ref{bl31} and the assumption that $U$ is abelian, we can find $\gamma>0$ with the following properties
\begin{enumerate}
\item We can devide $B(1)$ into small boxes $\{B_j\}$. For each $j$, there exists $y_j\in B(1)$ such that $B_j=y_jB(\gamma,\gamma,\dots,\gamma)$.
\item For each $j$, we have $\gamma<\eta(xa_{\ln T}y_j)/2$.
\item $\gamma\sim\eta(x a_{\ln T})$ and the implicit constant in $\sim$ depends only on $\Gamma\backslash G$.
\end{enumerate}
In fact, we can take such $\gamma$ by first taking the infimum of $\{\eta(xa_{\ln T}y)/2|y\in B(1)\}$ and then modifying it so that $1/\gamma$ is an integer. Note that the number of these boxes $B_j$ is $1/\gamma^n$. Now by Lemma \ref{bl33} we have
 
\begin{eqnarray*}
&{}&\left|\frac1{T^\alpha}\int_{B(T)}f(xu)du\right|=\left|\int_{B(1)}f(xa_{\ln T}ua_{-\ln T})du\right|\\
&\leq&\sum_j\left|\int_{B_j}f(xa_{\ln T}ua_{-\ln T})du\right|\\
&=&\sum_j\left|\int_{B(\gamma,\dots,\gamma)}f((xa_{\ln T}y_j)ua_{-\ln T})du\right|\\
&\ll&\frac1{\gamma^n}\frac1{T^a\gamma^b}\|f\|_{\infty,l}\ll\frac1{T^a\eta(xa_{\ln T})^{b+n}}\|f\|_{\infty,l}.
\end{eqnarray*}
This completes the proof of Theorem \ref{bequidistribution}.
\end{proof}

\begin{proof}[Proof of Theorem \ref{bequidistribution1}]
The proof is similar to that of Theorem \ref{bequidistribution}. We assume that $\int fd\mu=0$. We have 
$$\frac1{T^\alpha}\int_{U}f(xu)h(a_{-\ln T}ua_{\ln T})du=\int_Uf(xa_{\ln T}ua_{-\ln T})h(u)du=\int_Bf(xa_{\ln T}ua_{-\ln T})h(u)du$$ for some box $B\subset U$ since $h(u)$ is compactly supported. Using the same arguments as in the proof of Theorem \ref{bequidistribution}, we can find $\gamma>0$ with the following properties
\begin{enumerate}
\item We can devide $B$ into small boxes $\{B_j\}$. For each $j$, there exists $y_j\in B$ such that $B_j=y_jB(\gamma,\gamma,\dots,\gamma)$.
\item For each $j$, we have $\gamma<\eta(xa_{\ln T}y_j)/2$.
\item $\gamma\sim\eta(x a_{\ln T})$ and the implicit constant in $\sim$ depends only on $B$ and $\Gamma\backslash G$.
\end{enumerate}
By Lemma \ref{bl34}, we obtain that 
\begin{eqnarray*}
&{}&\left|\frac1{T^\alpha}\int_{U}f(xu)h(a_{-\ln T}ua_{\ln T})du\right|=\left|\int_Bf(xa_{\ln T}ua_{-\ln T})h(u)du\right|\\
&\leq&\sum_j\left|\int_{B_j}f(xa_{\ln T}ua_{-\ln T})h(u)du\right|\\
&=&\sum_j\left|\int_{B(\gamma,\dots,\gamma)}f((xa_{\ln T}y_j)ua_{-\ln T})h(y_ju)du\right|\\
&\ll&\sum_j\frac1{T^a\gamma^b}\|f\|_{\infty,l}\|h\|_{\infty,l}\ll\frac{\text{Vol}(B)}{\gamma^n}\frac1{T^a\gamma^b}\|f\|_{\infty,l}\|h\|_{\infty,l}\\
&\ll&\frac1{T^a\eta(xa_{\ln T})^{b+n}}\|f\|_{\infty,l}.
\end{eqnarray*}
Here the implicit constant depends on $h(u)$. This completes the proof of Theorem \ref{bequidistribution1}.
\end{proof}

\end{appendices}

\end{document}